\documentclass[a4paper]{amsart}

\usepackage[T1]{fontenc}
\usepackage{a4wide}
\usepackage{mathtools}
\usepackage{algorithm}
\usepackage{algpseudocode}
\usepackage{tikz}
\usepackage{subcaption}
\usepackage{diagbox}
\usepackage{multirow}
\usepackage{amsthm}
\usepackage{amssymb}
\usepackage[parfill]{parskip}
\usepackage{apptools}

\usepackage[sorting=none,url=false,maxbibnames=5,date=year]{biblatex}
\DeclareSourcemap{
    \maps{
        \map{ 
            \step[
                fieldsource=doi,
                match=\regexp{\{\\\_\}|\{\_\}|\\\_},
                replace=\regexp{\_}
            ]
        }
        \map{ 
            \step[
                fieldsource=url,
                match=\regexp{\{\\\_\}|\{\_\}|\\\_},
                replace=\regexp{\_}
            ]
        }
    }
}

\newtheorem{theorem}{Theorem}
\newtheorem{corollary}{Corollary}
\newtheorem{proposition}{Proposition}
\newtheorem{lemma}{Lemma}
\newtheorem{remark}{Remark}

\AtAppendix{\counterwithin{lemma}{section}}

\newcounter{assumptionCounter}

\newcounter{propertyCounter}

\newcommand{\ule}[1]{\mathrel{\underset{\clap{\text{\tiny #1}}}{\le}}}

\title[Carleman Linearization of Parabolic PDEs]{Carleman Linearization of Parabolic PDEs: Well-posedness, convergence, and efficient numerical methods}
\author{Bernhard Heinzelreiter and John W. Pearson}
\date{}

\begin{document}

\begin{abstract}
We explore how the analysis of the Carleman linearization can be extended to dynamical systems on infinite-dimensional Hilbert spaces with quadratic nonlinearities. We demonstrate the well-posedness and convergence of the truncated Carleman linearization under suitable assumptions on the dynamical system, which encompass common parabolic semi-linear partial differential equations such as the Navier--Stokes equations and nonlinear diffusion--advection--reaction equations. Upon discretization, we show that the total approximation error of the linearization decomposes into two independent components: the discretization error and the linearization error. This decomposition yields a convergence radius and convergence rate for the discretized linearization that are independent of the discretization. We thus justify the application of the linearization to parabolic PDE problems. Furthermore, it motivates the use of non-standard structure-exploiting numerical methods, such as sparse grids, taming the curse of dimensionality associated with the Carleman linearization. Finally, we verify the results with numerical experiments.
\end{abstract}

\maketitle

\section{Introduction}
% Linearization, Carleman, Koopman
% Carleman has been introduced by Carleman in ...
% Introduce quadratic (general) Cauchy problem (not too many details -> will discuss later)
% Basic idea (transform to bilinear system)
% Has been used in control theory, quantum computation applied to partial differential equations etc., cite papers in both applications
% Important to understand convergence
% Work has been done on error bounds for the finite-dimensional bounds (cite two papers), this case is well-understood
% These bounds depend on the logarithmic growth of A, the norm of B and the norm of the initial value, and can ultimately determine convergence or divergence
% When turning to finite element methods, one is forced to approximate the state and the operators by A_h, B_h, y_{0, h}. This means that the error bounds depend on a discretization parameter h. In equations such as the Burgers equation, B_h grows in h, leading to a smaller convergence radius and a slower convergence over all (theoretically)
% In (insert quantum citation) the case of the diffusion-reaction equation has been covered by analyzing the bounds and their dependence on h, this is not covered. This was observed for example in (Carleman applied to fluid paper)
% We first analyze the continuous case
% 
% As opposed to discretizing the system, and the applying the Carleman linearization, our approach delivers convergence radius independent of the discretization and to use alternative discretization techniques for the continuous truncated Carleman linearization
The linearization of nonlinear dynamical systems presents a significant challenge in applied mathematics. A key method for addressing this problem is the Carleman linearization, first introduced by Torsten Carleman in his seminal work \cite{Carleman1932ApplicationLineaires}. A variant of this method, known as truncated Carleman linearization or finite-section approximation of the Carleman linearization, has garnered considerable attention. This approach effectively transforms nonlinear systems into approximate linear systems while retaining some of the original nonlinear behavior, thus making analysis and computation more tractable.

The relevance of the Carleman linearization is underscored by its diverse applications across multiple domains. In applied mathematics, it has been employed in optimal control to simplify the handling of nonlinear constraints, making it easier to derive control laws \cite{Rauh2009CarlemanProcesses,Amini2019CarlemanSystems,Amini2020ApproximateEquation,Kar2023ReinforcementLinearization}. In the field of model order reduction, the linearization has been utilized to extend established order-reducing methods to complex nonlinear systems, thus enabling more efficient simulations and analyses \cite{Goyal2015ModelBilinearization,Benner2012Interpolation-basedSystems}. It has also found application in system identification problems; see, e.g., \cite{Abudia2023CarlemanBounds}. More recently, the linearization has proven valuable for quantum computing applications as it facilitates the development of quantum algorithms to solve large-scale nonlinear dynamical systems \cite{Surana2024AnSystems,Wu2025QuantumConditions}.

Since its generalization to nonlinear partial differential equations (PDEs) in \cite{Banks1992Infinite-dimensionalEquations}, the Carleman linearization has been successfully applied in practice to approximate solutions to nonlinear PDEs. For instance, the method has been employed in a range of fluid flow applications, from simplified benchmark problems such as the Burgers' equation (see, e.g., \cite{Benner2012Interpolation-basedSystems}) to more complex fluid flow problems (see, e.g., \cite{Sanavio2024ThreeFluids,Sanavio2024LatticeNumber}), as well as reaction--diffusion problems \cite{Liu2023EfficientEstimation}. These and related applications highlight the potential of the linearization technique as a powerful tool for addressing the complexities inherent in nonlinear infinite-dimensional dynamical systems.

A critical question in the application of the Carleman linearization is its accuracy in approximating the original nonlinear system, particularly the effect of truncation on this accuracy. This question is of paramount importance as it determines how well the linearized system can recover the nonlinear behavior of the original system. Error estimates of the truncated linearization have recently been established for finite-dimensional dynamical systems, i.e., nonlinear ordinary differential equations (ODEs). For instance, \cite{Forets2017ExplicitLinearization} provides comprehensive error estimates for quadratically nonlinear equations, while \cite{Amini2021ErrorSystems,Amini2025CarlemanApproximations} extend these estimates to general nonlinear systems. Although these theoretical estimates generally align well with experimental observations for ODEs, they cannot fully explain the behavior of the linearization when applied to certain infinite-dimensional dynamical systems, including a range of PDEs.
A straightforward method to generalize the error estimates for PDEs involves first discretizing the equation to arrive at a finite-dimensional nonlinear dynamical system, which is then linearized, followed by applying the available error estimates from \cite{Amini2025CarlemanApproximations}. The primary issue in such case is that if the PDE is first discretized and then linearized, the existing error estimates depend on the properties of the discrete operators. For certain classes of nonlinearities, these bounds can be shown to be robust with respect to the discretization; see, e.g., \cite{Liu2023EfficientEstimation} for nonlinear reaction terms. In the more general case, however, this dependency can lead to error estimates that are not robust, a problem first observed by \cite{Gonzalez-Conde2025QuantumDynamics} in the context of fluid flow problems. Despite this, it has been empirically observed that the linearization can approximate nonlinear behavior effectively, even in the context of infinite-dimensional systems and independently of the discretization. An additional challenge for infinite-dimensional systems that has not been addressed in existing literature is proving that the resulting linearization is well-posed.

In this paper, we tackle the questions of well-posedness and convergence of the truncated Carleman linearization when applied to a class of infinite-dimensional parabolic systems. We establish both well-posedness and convergence in an infinite-dimensional and undiscretized setting, which allows for error estimates that are independent of any discretization of the system. This approach not only provides a robust theoretical framework and better understanding of the linearization but also offers practical advantages in numerical applications. We demonstrate that this approach enables the separation of the linearization error from the discretization error through a so-called \textit{linearize-then-discretize} strategy, ultimately allowing us to develop new numerical approximations that mitigate the curse of dimensionality associated with the linearization. This paves the way for a new class of efficient and accurate numerical methods, essential for the practical implementation of the linearization in real-world problems.

The structure of the paper is as follows: In Section~\ref{sec:NonlinearCauchyProblem}, we introduce the nonlinear dynamical system and outline the appropriate assumptions on the state space and the involved operators to guarantee the well-posedness of the nonlinear system. The problem setup covers a range of parabolic second-order differential equations but is not limited to these. Section~\ref{sec:TruncatedCL} offers a non-rigorous introduction to the concept of the truncated Carleman linearization, providing an intuitive understanding of the method. In Section~\ref{sec:Wellposedness}, we equip the Carleman linearization with appropriate function spaces and demonstrate the well-posedness of the problem under additional assumptions. This section is crucial for establishing the theoretical validity of the linearization method in an infinite-dimensional setting. Section~\ref{sec:Convergence}
shows how further assumptions ensure the convergence of the linearization and showcases certain standard systems that fulfill these assumptions, thereby illustrating the practical applicability of the theoretical results. Section~\ref{sec:NumericalMethods} discusses how our approach motivates the development of efficient numerical methods. It is shown how the traditional Carleman linearization of PDEs, which is based on a discretization of the equation, is a special case of a greater class of numerical methods. Furthermore, this section highlights how our approach overcomes the shortcomings of the existing error bounds. The theoretical results are verified in Section~\ref{sec:NumericalExperiments} through a series of numerical experiments, which provide empirical evidence supporting the validity and effectiveness of the proposed convergence results and methods. Finally, Section~\ref{sec:Conclusion} concludes the work and provides an outlook on potential future research directions, suggesting avenues for further exploration of hyperbolic equations and efficient numerical methods.

\section{Nonlinear Cauchy Problems in Hilbert Spaces}
\label{sec:NonlinearCauchyProblem}

We are interested in the analysis of nonlinear Cauchy problems defined on an infinite-dimensional separable Hilbert space $H$. Specifically, we consider systems with quadratic nonlinearities, i.e., of the form
\begin{equation}
\begin{aligned}
    y'(t) + A y(t) + B (y(t)\otimes y(t)) &= f(t) \quad\text{for a.e. } t \in (0, T), \\
    y(0) &= y_0,
\end{aligned}
    \label{eq:CauchyProblemStrong}
\end{equation}
where we seek a solution $y$, given some initial condition $y_0$ and forcing $f$. The final time can be bounded or unbounded, i.e., $T \in (0, \infty]$. The system is considered in the space $H$. We refer to $A$ as the linear part of the dynamical system. It is assumed that $A: D(A) \to H$ is a closed invertible operator with a dense domain $D(A)$ continuously embedded in $H$. Moreover, $A$ generates an analytic semigroup. The operator $B: D(A) \times D(A) \to H$ is assumed continuous and is referred to as the quadratic part of the system. The forcing lives in the space $L^2(0, T; H)$. If the equalities in system~\eqref{eq:CauchyProblemStrong} are considered in $H$ and $y_0 \in D(A)$, the dynamical system is called the strong form of the Cauchy problem. 

\subsection{Weak Formulation}
Since our ultimate goal is the analysis of PDEs and their efficient numerical approximation, it is more suitable to consider the problem in its weak form. The following analysis presents the key ideas of the theory of parabolic systems and refines the assumptions on the present operators. For an in-depth discussion of the topic, the reader is referred to \cite{Bensoussan2007RepresentationSystems}. Let $V = [D(A), H]_{1/2}$ be the interpolation space between $D(A)$ and $H$, and $V'$ be its topological dual. Let $(\cdot, \cdot)_H$ be the inner product of $H$ and $\langle \cdot, \cdot \rangle_V$ the duality pairing of $V' \times V$. We assume that the embedding $V \hookrightarrow H \cong H' \hookrightarrow V'$ is continuous and dense, where each element $h \in H$ is identified with an element of $H'$ via the mapping $g \mapsto \langle h, g\rangle_V = (h, g)_H$ for all $g \in H$. Then, $(V, H, V')$ forms a Gelfand triple. Assume that $A$ and $B$ can be extended to $A: V \to V'$ and $B: V \times V \to V'$. Then, the weak form of the Cauchy problem reads as
\begin{equation}
    \begin{aligned}
    y'(t) + A y(t) + B (y(t) \otimes y(t)) &= f(t) \quad\text{in } L^2(0, T; V'), \\
    y(0) &= y_0,
    \end{aligned}
    \label{eq:CauchyProblemWeak}
\end{equation}
where the equality is to be considered in the dual space $L^2(0, T; V')$ and $y_0 \in H$, and the solution is sought in the space $y \in W(0, T; V, V')$ with
\begin{align*}
    W(0, T; X, Y) &= \left\{ y \in L^2(0, T; X) \mid y' \in L^2(0, T; Y) \right\}, \\
    \| y \|_{W(0, T; X, Y)}^2 &= \| y \|_{L^2(0, T; X)}^2 + \| y' \|_{L^2(0, T; Y)}^2
\end{align*}
for any pair of Banach spaces $X$ and $Y$. Due to the Lions--Magenes lemma \cite{Lions1972Non-HomogeneousApplications}, the space of continuous functions in $H$ is continuously embedded in $W(0, T; V, V')$, i.e., $C(0, T; H) \hookrightarrow W(0, T; V, V')$, which makes point-wise evaluations well-posed. The weak formulation allows us to impose weaker regularity on the forcing $f \in L^2(0, T; V')$.

\subsection{Linear Equation}
We assume that $V$ and $V'$ admit an eigenvalue representation. For that, assume there exists a closed positive self-adjoint operator $L$ with domain $D(L)$ such that $D(L^{\bar{\alpha}/2}) = D(A^{\bar{\alpha}/2})$ for some $\bar{\alpha} \ge 1$. This allows us to define the norm of $V$ by the equivalent norm $\| \cdot \|_V := \|L^{1/2} \cdot\|_H$. Since $L$ is self-adjoint, we can identify this norm and, hence, the space $V$ with its eigenvectors and eigenvalues. Let $\{\varphi_i\}_{i \in \mathbb{N}}$ be eigenvectors of $L$ forming an orthonormal basis of $H$ and $\{\lambda_i\}_{i \in \mathbb{N}}$ its corresponding eigenvalues. Then, we can rewrite the $V$-norm as
\begin{align*}
    \| v \|_V^2 = \sum_{i = 1}^\infty \lambda_i \left( v, \varphi_i \right)_H^2.
\end{align*}
Assume that $A$ is bounded, i.e.,
\begin{align}
\refstepcounter{assumptionCounter}
\tag{A\theassumptionCounter}
\label{assumption:ABounded}
    \left| \langle A v, w\rangle_V \right| \le \beta \|v\|_V \|w \|_V.
\end{align}
Furthermore, let $A$ be $V$-$H$ coercive (sometimes referred to as weakly coercive), which is achieved if there exist constants $\lambda \ge 0$ and $\gamma > 0$ with
\begin{align}
\refstepcounter{assumptionCounter}
\tag{A\theassumptionCounter}
\label{assumption:AVHCoercive}
\langle A v, v\rangle_V + \lambda \|v\|^2_H \ge \gamma \|v\|_V^2 \quad \text{for all } v \in V.
\end{align}
This implies that the linear equivalent of our Cauchy problem is well-posed.
\begin{lemma}
\label{lem:WellPosednessLinear}
    Let $T < \infty$ and $A(t)$ be a family of $V$-$H$ operators fulfilling \eqref{assumption:ABounded} and \eqref{assumption:AVHCoercive} for $t \in (0, T)$ with constants $\beta$, $\gamma$, and $\lambda$ independent of $t$. Then, for every $y_0 \in H$ and $f \in L^2(0, T; V')$, the linear Cauchy problem
    \begin{equation*}
    \begin{aligned}
    y'(t) + A(t) y(t) &= f(t) \quad\text{in } L^2(0, T; V'), \\
    y(0) &= y_0
    \end{aligned}
    %\label{eq:CauchyProblemWeakLinear}
    \end{equation*}
has a unique solution $y \in W(0, T; V, V')$. This solution satisfies the estimate
\begin{align}
\label{eq:ParabolicEstimate}
    \| y(t) \|_H^2 + \gamma \int_0^t \exp(2 \lambda (t - \tau)) \| y(\tau)\|^2_V \mathrm{d}\tau \le \exp(2 \lambda t)\| y_0\|_H^2 + \frac{1}{\gamma} \int_0^t \exp(2 \lambda (t - \tau)) \| f(\tau)\|_{V'}^2 \mathrm{d}\tau
\end{align}
for all $t \in [0, T)$. Moreover, there exists a constant $c_L \ge 1$, independent of $y_0$, $f$, and $T$, such that
\begin{align}
\label{eq:ParabolicEstimateWNorm}
    \|y\|_{L^\infty(0, T; H)} + \|y\|_{W(0, T; V, V')} \le c_L \exp(\lambda T) \left( \|y_0 \|_H + \| f \|_{L^2(0, T; V')} \right).
\end{align}
If $\lambda = 0$, then one can choose $T = \infty$.
\end{lemma}
\begin{proof}
    See Theorem II.2.1.1 in \cite{Bensoussan2007RepresentationSystems} for the uniqueness and existence, and \cite{Zeidler1990II/Operators,Quarteroni1994NumericalEquations} for the estimate \eqref{eq:ParabolicEstimate}. The inequality \eqref{eq:ParabolicEstimateWNorm} is a direct consequence of \eqref{eq:ParabolicEstimate} and the properties of $A(t)$.
\end{proof}
While the linear operator $A$ in the original Cauchy problem is time-invariant and fulfills all the assumptions of Lemma~\ref{lem:WellPosednessLinear}, the result covers time-dependent linear operators $A(t)$. This will be leveraged to prove the Carleman linearization's well-posedness and convergence. Additionally, the result applies not only to the spaces $V$ and $H$, but also to other suitable spaces forming Gelfand triples with the same properties. In this case, the concepts of boundedness and coercivity should be understood in their respective spaces. 

The system remains well-posed even under weak, time-dependent perturbations, as the following lemma shows.
\begin{lemma}
\label{lem:PerturbationLinear}
    Let $T < \infty$ and $A(t)$ be a family of operators fulfilling the same assumptions as in Lemma~\ref{lem:WellPosednessLinear}. Let $K(t)$ be a family of operators
    \begin{align*}
        K(t): V \to H \quad \text{for a.e. } t \in (0, T)
    \end{align*}
    such that for all $v \in V$ the mapping $t \mapsto K(t) v$ belongs to $L^\infty(0, T; H)$.
    Then, $A(t) + K(t)$ is uniformly $V$-$H$ coercive and the variational Cauchy problem
    \begin{equation*}
        \begin{aligned}
        y'(t) + (A(t) + K(t)) y(t) &= f(t) \quad\text{in } L^2(0, T; V'), \\
        y(0) &= y_0
        \end{aligned}
    \end{equation*}
    has a unique solution in $W(0, T; V, V')$ for all $y_0 \in H$ and $f \in L^2(0, T; V')$ that depends continuously on the data.
\end{lemma}
\begin{proof}
    See Theorem II.2.1.2 in \cite{Bensoussan2007RepresentationSystems} with $\theta = 0$.
\end{proof}

\subsection{Nonlinear Equation}
Moving to the nonlinear equation, we have to add assumptions on $B$ to show the well-posedness of the Cauchy problem. While there are various suitable types of assumptions, we focus on a particular choice that covers most common PDE problems. Assume there is a constant $c_B$ such that
\begin{align}
\refstepcounter{assumptionCounter}
\tag{A\theassumptionCounter}
\label{assumption:BBounded}
    \left| \langle B(u \otimes v), w\rangle_V \right| \le c_B \| u \|^{\frac{1}{2}}_H \| u \|^{\frac{1}{2}}_V \| v \|^{\frac{1}{2}}_H \| v \|^{\frac{1}{2}}_V \| w \|_V \quad \text{for all } u, v, w \in V.
\end{align}
The following lemma is a modification of Lemma~5 in \cite{Breiten2019FeedbackApproximation} and proves the local existence and uniqueness of solutions to the nonlinear Cauchy problem on finite time intervals for unstable $A$ and on unbounded time intervals for the stable case. This adaptation also takes into account the explicit dependence of the constants on $T$, which will play a crucial role in the convergence behavior of the Carleman linearization.
\begin{lemma}
\label{lem:WellPosednessNonlinear}
    Let $T < \infty$, and $A$ and $B$ fulfill \eqref{assumption:ABounded}, \eqref{assumption:AVHCoercive}, and \eqref{assumption:BBounded}. There exists a constant $c_N$ independent of $T$ and $c_B$ such that for all $y_0 \in H$ and $f \in L^2(0, T; V')$ with
    \begin{align*}
        \| y_0 \|_H + \|f\|_{L^2(0, T; V')} \le \rho(T) := \frac{1}{8 c_N^2 \exp(2 \lambda T) c_B},
    \end{align*}
    there exists a unique solution $y \in W(0, T; V, V')$ to the nonlinear Cauchy problem \eqref{eq:CauchyProblemWeak}. Moreover, the estimate
    \begin{align*}
        \|y_0\|_H + \| y \|_{W(0, T; V, V')} \le 2 c_N \exp(\lambda T) \left(\| y_0 \|_H + \|f\|_{L^2(0, T; V')}\right)
    \end{align*}
    holds. If the coercivity constant $\lambda$ of the linear operator is zero, then the same result holds for $T = \infty$.
\end{lemma}
\begin{proof}
The proof is provided in Appendix~\ref{appendix:ProofsNonlinearCauchyProblem}.
\end{proof}

\section{Truncated Carleman Linearization}
\label{sec:TruncatedCL}
We start with an informal definition of the Carleman linearization, focusing on the concept and leaving a rigorous mathematical analysis for the subsequent sections. The main idea is to lift the dynamical system into higher-dimensional tensor products of the original solution space. While this increases the dimensionality of the spaces, the resulting system is linear. Thus, the linearization can be seen as a trade-off between the nonlinearity and the complexity of the underlying spaces. 

We first introduce the operators necessary for the linearization. Define the Kronecker sum of an operator or function $T$ by
\begin{align*}
    \bigoplus^k T := \sum_{i = 1}^k \left( \bigotimes^{i - 1} I \right) \otimes T \otimes \left( \bigotimes^{k - i} I \right),
\end{align*}
where $I$ denotes the identity operator.
This allows us to define
\begin{align*}
    A_k = \bigoplus^k A, \quad B_k = \bigoplus^k B, \quad F_k(t) = \bigoplus^k (-f(t)).
\end{align*}
To illustrate the structure and action of these operators, consider $u, v, w \in V$. Then, for $k = 2$, we obtain
\begin{align*}
    A_2 (u \otimes v) &= (Au) \otimes v + u \otimes (Av), \\
    B_2 (u \otimes v \otimes w) &= (B (u \otimes v)) \otimes w + u \otimes (B (v \otimes w)),\\
    F_2(t) u &= (-f(t)) \otimes u + u \otimes (-f(t)).
\end{align*}
If a solution $y$ to \eqref{eq:CauchyProblemStrong} is sufficiently regular, then its moments $y^{(k)}(t) := \bigotimes^k y(t)$ fulfill the equation
\begin{align*}
    \frac{\mathrm{d}}{\mathrm{d}t} y^{(1)}(t) + A_1 y^{(1)}(t) + B_1 y^{(2)}(t) &= f(t), \\
    \frac{\mathrm{d}}{\mathrm{d}t} y^{(k)}(t) + F_k(t) y^{(k - 1)}(t) + A_k y^{(k)}(t) + B_k y^{(k + 1)}(t) &= 0 &\text{for } k > 1, \\
    y^{(k)}(0) &= y_0^{(k)} &\text{for } k \ge 1.
\end{align*}
These equations establish a linear relationship between moments of order $k - 1$, $k$, and $k + 1$. Instead of solving for $y$, we can seek a solution consisting of moments $y^{(k)}$ that fulfill the above dynamical system. While this erases the nonlinearity from the dynamical system, the approach has a caveat. To fully describe the equation for $y^{(k)}$, one needs $y^{(k + 1)}$. This, in turn, requires one to add an equation for $y^{(k + 1)}$, ultimately resulting in an infinite chain of equations, which is called the Carleman linearization. It is not obvious that the infinite system of equations is well-posed; neither is it clear whether a solution of this chain solves the original nonlinear Cauchy problem. For self-adjoint positive $A$ and $f = 0$ some of these questions were answered in, e.g., \cite{Vishik1988MathematicalHydromechanics,Fursikov1989OnSystem,Fursikov1993MomentSide} in the scope of fluid flow problems, where the chain is referred to as the \textit{Friedman--Keller chain of equations} and arises from a statistical analysis of Navier--Stokes equations. Here, we are interested in the computation and approximation aspects of the linearization and will, therefore, omit an analysis of the infinite system. Instead, we explore an approach to circumvent the infinite sequence, restricting our analysis to its finite equivalent.

In order to arrive at a computable dynamical system, we have to close the infinite sequence of equations, sometimes referred to as \textit{moment closure}. For the Carleman linearization, the most common approach is to truncate the system and set $y^{(N + 1)}(t) = 0$ for a fixed $N > 1$, decoupling the first $N$ equations. The rationale behind this approach is that higher-order terms are expected to become negligible for the first moment. This results in the so-called \textit{truncated Carleman linearization}, or sometimes also referred to as its \textit{finite-section approximation}. Then, the dynamical system can be written as
\begin{equation}
\begin{aligned}
    y_N'(t) + \mathcal{A}_N(t) y_N(t) &= f_N(t) \quad \text{ for } t \in (0, T),\\
    y_N(0) &= y_{N, 0} = \begin{pmatrix} y_0^{(1)} & y_0^{(2)} & \cdots & y_0^{(N)} \end{pmatrix}^T, \\
\end{aligned}
\tag{TC}
\label{eq:TruncatedCarleman}
\end{equation}
with the block operator matrix and right-hand side defined as
\begin{align*}
    \mathcal{A}_N(t) &= \begin{pmatrix}
        A_1 & B_1 & 0 & \cdots & 0 \\
        F_2(t) & A_2 & B_2 & \ddots & \vdots \\
        0 & \ddots & \ddots & \ddots & 0 \\
        \vdots & \ddots & F_{N - 1}(t) & A_{N - 1} & B_{N - 1} \\
        0 & \cdots & 0 & F_N(t) & A_N
    \end{pmatrix}, \quad
    f_N(t) = \begin{pmatrix} f(t) \\ 0 \\ \vdots \\ 0 \end{pmatrix}.
\end{align*}
We denote the number $N$ as the \textit{truncation level}.

\section{Well-Posedness of the Truncated Carleman Linearization}
\label{sec:Wellposedness}
To rigorously define the system~\eqref{eq:TruncatedCarleman}, we introduce suitable function spaces and impose specific assumptions on $B$ and $f$, which build upon and extend the approaches of \cite{Vishik1988MathematicalHydromechanics,Fursikov1989OnSystem,Fursikov1993MomentSide}. These criteria will ultimately enable us to show existence and uniqueness of the dynamical system, i.e., its well-posedness. This is done by establishing a weak formulation of the truncated linearization.

\subsection{Function Spaces}
Let us define the family of interpolation spaces with the corresponding norms
\begin{align*}
\begin{alignedat}{3}
    V^\alpha &:= [D(L), H]_{1 - \alpha / 2}, &\quad \| v \|_{V^\alpha}^2 &:= \| L^{\alpha/2} v\|_H^2 = \sum_{i = 1}^\infty \lambda_i^\alpha (v, \varphi_i)_H^2, \\
    V^{-\alpha} &:= \left(V^\alpha\right)', &\quad \|v\|_{V^{-\alpha}}^2 &:= \sup_{0 \neq w \in V^\alpha} \frac{\langle v, w\rangle_{V^{-\alpha}}^2}{\|w \|_{V^\alpha}^2} = \sum_{i = 1}^\infty \lambda_i^{-\alpha} \langle v, \varphi_i\rangle_V^2
\end{alignedat}
\end{align*}
for $\alpha \ge 0$. This means that $V^0 = H$ and $V^1 = V$. For example, in the case of the Sobolev spaces $H = L^2(\Omega)$ and $V = H^1_0(\Omega)$ for a bounded Lipschitz domain $\Omega \subset \mathbb{R}^n$, these spaces can be identified with (cf.~\cite[pp.~283 ff.]{Lasiecka2000ControlEquations})
\begin{align*}
    V^\alpha = \begin{cases}
        H^{\alpha}(\Omega) & \text{for } 0 \le \alpha < \frac{1}{2}, \\
        H^{\alpha}(\Omega) \cap H^1_0(\Omega) & \text{for } \frac{1}{2} \le \alpha \le 2.
    \end{cases}
\end{align*}
Furthermore, define the tensor product spaces for $k \in \mathbb{N}$ and $q \in [-1, 1]$
\begin{align*}
    H(k) = \bigotimes_{j = 1}^k H, \quad V^\alpha(k) = \bigotimes_{j = 1}^k V^\alpha, \quad    V^\alpha_q(k) = \bigcap_{i = 1}^k \bigotimes_{j = 1}^k V^{\alpha + \delta_{i, j} q},
\end{align*}
where $V^\alpha_q(k)$ can be interpreted as the tensor product space $V^\alpha(k)$ with increased ($q > 0$) or decreased ($q < 0$) regularity in single directions. The space $H(k)$ is a Hilbert space endowed with the standard tensor product scalar product. Note that $V^\alpha(k) = V^\alpha_0(k)$ and $H(k) = V^0(k)$. The corresponding norms are defined as
\begin{align*}
    \| v \|_{V_q^\alpha(k)}^2 &:= \sum_{\vec{i} \in \mathbb{N}^k} \pi(\lambda_{\vec{i}})^\alpha \sigma(\lambda_{\vec{i}})^q \langle v, \varphi_{\vec{i}} \rangle^2_V.
    %\\
    % \text{especially: }\| v_1 \otimes \cdots \otimes v_k \|_{V^0_1(k)}^2 &= \sum_{i = 1}^k \prod_{j = 1}^k \| v_j \|_{V^{\delta_{i, j}}}^2.
\end{align*}
The function $\varphi_{\vec{i}}$ is defined as $\varphi_{\vec{i}} = \varphi_{\vec{i}_1} \otimes \cdots \otimes \varphi_{\vec{i}_k}$ for any multi-index $\vec{i} \in \mathbb{N}^k$ and the associated tuple of eigenvalues as $\lambda_{\vec{i}} = (\lambda_{\vec{i}_1}, \cdots, \lambda_{\vec{i}_k})$. Furthermore, we set $\pi(\lambda_{\vec{i}}) := \prod_{l = 1}^k \lambda_{\vec{i}_l}$ and $\sigma(\lambda_{\vec{i}}) := \sum_{l = 1}^k \lambda_{\vec{i}_l}$.
In the special case of $k = 1$,
we have that $V^\alpha_q(1) = V^{\alpha + q}$. The dual space $(V_{1}^0(k))'$ can be identified with $V_{-1}^0(k)$ and their norms coincide.

Finally, we provide a function space for a complete solution $y_N$ to \eqref{eq:TruncatedCarleman}, which assembles moments in $V^\alpha_q(k)$ into a vector of moments $y_N$. For that, define the direct sums of vector spaces
\begin{align*}
\begin{alignedat}{3}
U_q^\alpha(N) &:= \bigoplus_{k = 1}^N V_q^\alpha(k), \quad &\| (v_1, \dots, v_N) \|_{U_q^\alpha(N)}^2 &= \sum_{k = 1}^N \|v_k\|_{V_q^\alpha(k)}^2, \\
Y(N) &:= \bigoplus_{k = 1}^N H(k), \quad &\| (v_1, \dots, v_N) \|_{Y(N)}^2 &= \sum_{k = 1}^N \|v_k\|_{H(k)}^2.
\end{alignedat}
\end{align*}
The space $Y(N)$ forms a Hilbert space with the scalar product $( (u_1, \dots, u_N), (v_1, \dots, v_N))_{Y(N)} = \sum_{k = 1}^N ( u_k, v_k)_{H(k)}$. It is noted that $(U_1^0(N), Y(N), U_{-1}^0(N))$ forms a Gelfand triple as the dual space $(U_1^0(N))'$ can be identified with $U_{-1}^0(N)$.

\subsection{Properties of $A_k$, $B_k$, and $F_k$}
The operators $A_k$, $B_k$, and $F_k$ are essential components of the truncated Carleman linearization. The construction of the operators through the Kronecker sum suggests that the operators' norms grow at least linearly in $k$. However, the dependence of the boundedness and the coercivity on $k$ can be refined and, in certain cases, eliminated, which will play a crucial role in proving the convergence of the linearization.

The operator $A_k$ is a closed operator with domain $D(A_k) = V_2^0(k)$. From the definition of the space $V_1^0(k)$, we know that $A_k$ can be extended to $A_k: V_1^0(k) \to V_{-1}^0(k)$. Similarly to above, the function space of choice for a weak formulation is $D(A_k^{1/2})$.
Theorem~13.1 in \cite{Lions1972Non-HomogeneousApplications} combined with the properties of $L$ implies that we can identify the interpolation space as $D(A_k^{1/2}) = V^0_1(k)$, and, therefore, $D(A_k^{1/2})' = V^0_{-1}(k)$. Thus, the operator $A_k: D(A_k^{1/2}) \to D(A_k^{1/2})'$ is well defined.
The following lemma demonstrates that not only can the Kronecker sum of $A$ be extended to the interpolation space, but it also maintains coercivity and boundedness.
\begin{lemma}
\label{lem:Ak}
Assume the operator $A$ fulfills \eqref{assumption:ABounded} and \eqref{assumption:AVHCoercive} with constants $\beta$, $\gamma$, and $\lambda$. Then, the operator $A_k: V_1^0(k) \to V_{-1}^0(k)$ is bounded and $V_1^0(k)$-$H(k)$ coercive with constants
\begin{align*}
    \langle A_k u, v \rangle_{V_1^0(k)} &\le \beta \|u\|_{V_1^0(k)} \|v\|_{V_1^0(k)},  \\
    \langle A_k v, v\rangle_{V_1^0(k)} + k \lambda \| v \|^2_{H(k)} &\ge \gamma \| v \|_{V_1^0(k)}^2
\end{align*}
for all $u, v \in V_1^0(k)$.
\end{lemma}
\begin{proof}
The proof is provided in Appendix~\ref{appendix:ProofsAB}.
\end{proof}
While boundedness and coercivity are pivotal properties for showing existence and uniqueness of parabolic problems on its own, the result can even be strengthened in the case $\lambda = 0$, in which case the constants for boundedness and coercivity are independent of $k$.

Turning to the operator $B_k$, assume that $B$ is a symmetric and continuous bilinear mapping between the spaces $B: V^\alpha_1(2) \to V^{\alpha - 1 + \varepsilon}$ for fixed constants $\varepsilon \in [0, 1]$ and $\alpha \ge 0$. Let $c_{B}(\alpha, \varepsilon) > 0$ be a constant such that
\begin{align}
\refstepcounter{assumptionCounter}
\tag{A\theassumptionCounter}
\label{assumption:BBilinearMapping}
\| B v \|_{V^{\alpha - 1 + \varepsilon}} \le c_{B}(\alpha, \varepsilon) \|v \|_{V^\alpha_1(2)} \quad \text{for all } v \in V^\alpha_1(2).
\end{align}
It is noted that property \eqref{assumption:BBounded} implies \eqref{assumption:BBilinearMapping} for $\alpha = 0$ and $\varepsilon = 0$ (by using Young's inequality). However, the converse does not hold in general, i.e., \eqref{assumption:BBounded} cannot be deduced from \eqref{assumption:BBilinearMapping}, and, therefore, is a stronger condition on $B$. The importance of the distinction between these assumptions becomes evident in the subsequent discussion of operator examples.
The following result for $B_k$ was shown in \cite{Fursikov1989OnSystem,Fursikov1993MomentSide}.
\begin{lemma}
\label{lem:Bk}
    Let $B$ fulfill \eqref{assumption:BBilinearMapping} with constants $\alpha$ and $\varepsilon$. The operator $B_k: V^\alpha_1(k + 1) \to V^\alpha_{-1 + \varepsilon}(k)$ is well-defined and continuous. More specifically,
    \begin{align*}
        \| B_k v \|_{V_{-1 + \varepsilon}^\alpha(k)} \le \sqrt{2}c_{B}(\alpha, \varepsilon) k^{\varepsilon / 2} \| v \|_{V_{1}^\alpha(k + 1)} \quad\text{for all } v \in V_1^\alpha(k + 1),
    \end{align*}
    where $c_B(\alpha, \varepsilon)$ is the constant from \eqref{assumption:BBilinearMapping} and does not depend on $k$. 
\end{lemma}
\begin{proof}
    The proof is provided in Appendix~\ref{appendix:ProofsAB} and is a generalization of the proof of Lemma 2.4 in \cite{Fursikov1989OnSystem} (case $\varepsilon = 0$) to the more general result for $\varepsilon \in [0, 1]$ stated in \cite{Fursikov1993MomentSide} (proof not provided therein).
\end{proof}

We can show similar bounds for $F_k$.
\begin{lemma}
    \label{lem:Fk}
    For any $\alpha \ge 0$, the operator $F_k(t): V^\alpha_1(k) \to V_{-1 + \varepsilon}^\alpha(k + 1)$ is bounded for $f(t) \in V^\alpha$. More specifically, it holds that
    \begin{align*}
        \|F_k(t) v\|_{V^\alpha_{-1 + \varepsilon}(k + 1)} \le c_F(\varepsilon) \|f(t)\|_{V^\alpha} k^{\varepsilon / 2} \|v\|_{V_1^\alpha(k)} \quad \text{for all } v \in V_1^\alpha(k),
    \end{align*}
    where the constant $c_F(\varepsilon)$ does not depend on $f$ or $k$.
\end{lemma}
\begin{proof}
    The proof is provided in Appendix~\ref{appendix:ProofsAB}.
\end{proof}
% \begin{remark}
%     It is noted that the constants in the estimates from Lemmas~\ref{lem:Bk} and \ref{lem:Fk} depend on the coercivity constant $\gamma$ of $A_s$. As $\gamma$ grows, the continuity constants of $B_k$ and $F_k$ decrease. This means that the stability of the operator $A_s$ and, therefore, $A$ has an immediate effect on the norms of $B_k$ and $F_k$. As we will see later, smaller norms of $B_k$ and $F_k$ lead to a better convergence behavior of the Carleman linearization and to a larger radius of convergence. Such result aligns with the findings in \cite{Forets2017ExplicitLinearization,Amini2025CarlemanApproximations} for finite-dimensional systems.
% \end{remark}

\subsection{Lifted Linearization and Well-Posedness}
% Briefly discuss that \alpha > 0 might not be too great, thus we
% We could define the system in a nonstandard scalar product, however, we choose to transform the system first
% Transform the original truncated linearization

Intuitively, we would like to consider system~\eqref{eq:TruncatedCarleman} in the space $U_1^0(N)$, meaning we seek a solution $y_N \in U_1^0(N)$ while the equality is to be considered in its dual space $U_{-1}^0(N)$. However, if assumption~\eqref{assumption:BBilinearMapping} does not hold for $\alpha = 0$ but for some $\alpha > 0$, Lemma~\ref{lem:Bk} does not allow us to infer well-posedness of the linearization through Lemma~\ref{lem:PerturbationLinear} in the classical weak formulation. One way to address this issue is to consider the equation directly in $U_1^\alpha(N)$ and $U_{-1}^\alpha(N)$ instead, as has been done in \cite{Fursikov1989OnSystem,Fursikov1993MomentSide} in the case of self-adjoint positive $A$ and $f = 0$. Here, we choose a different approach: We first transform the original linearization into an alternative dynamical system, which restricts the solution to higher regularity. This, ultimately, enables us to apply Lemma~\ref{lem:PerturbationLinear}.
% and is linked to the notion of so-called strong variational solutions

First, define the operator $A_\lambda := A + \lambda I$ where $I$ denotes the identity operator and $\lambda$ the constant from \eqref{assumption:AVHCoercive}. Then, $A_\lambda$ is $V$-$H$ coercive with the same constant $\gamma$ as $A$ and $\lambda = 0$. Since $A_\lambda$ is a positive operator, its fractional powers are well-defined (see, e.g., Section~2.6 in \cite{Pazy1983SemigroupsEquations}), and we can define $\bar{A}_{\lambda, k}^\xi := \bigotimes^k A_{\lambda}^\xi$ for any $\xi \in \mathbb{R}$. Furthermore, $A_\lambda^\xi$ commutes with $A$ and, consequently, $\bar{A}^\xi_{\lambda, k}$ commutes with $A_k$.

Instead of seeking a solution $y_N$ to \eqref{eq:TruncatedCarleman}, we seek a solution $\tilde{y}_N(t) = \mathcal{D}_{\lambda, N}^{\alpha} y_N(t)$ with $\mathcal{D}_{\lambda, N}^\xi := \operatorname{diag}(\bar{A}^{\xi/2}_{\lambda, 1}, \dots, \bar{A}^{\xi/2}_{\lambda, N})$ for $\xi \in \mathbb{R}$, where we are specifically interested in the case $\alpha \in [0, \bar{\alpha} - 1]$. Plugging this into the original linearization leads to the dynamical system
\begin{align}
\begin{aligned}
    \tilde{y}'_N(t) &= \underbrace{\mathcal{D}_{\lambda, N}^{\alpha} \mathcal{A}_N(t) \mathcal{D}_{\lambda, N}^{-\alpha}}_{\widetilde{\mathcal{A}}_N(t)} \tilde{y}_N(t) + \underbrace{\mathcal{D}_{\lambda, N}^{\alpha} f_N(t)}_{\tilde{f}_N(t)}, \\
    \tilde{y}_N(0) &= \underbrace{\mathcal{D}_{\lambda, N}^{\alpha} y_{N, 0}}_{\tilde{y}_{N, 0}}.
\end{aligned}
\tag{$\text{TC}_\alpha$}
\label{eq:TransformedCarleman}
\end{align}
The block operator matrix $\widetilde{\mathcal{A}}_N$, the forcing $\tilde{f}_N$, and the initial condition have a similar form as in \eqref{eq:TruncatedCarleman}:
\begin{gather*}
    \widetilde{\mathcal{A}}_N(t) = \begin{pmatrix}
        A_1 & \widetilde{B}_1 & 0 & \cdots & 0 \\
        \widetilde{F}_2(t) & A_2 & \widetilde{B}_2 & \ddots & \vdots \\
        0 & \ddots & \ddots & \ddots & 0 \\
        \vdots & \ddots & \widetilde{F}_{N - 1}(t) & A_{N - 1} & \widetilde{B}_{N - 1} \\
        0 & \cdots & 0 & \widetilde{F}_N(t) & A_N
    \end{pmatrix}, \\
    \tilde{f}_N(t) = \begin{pmatrix} A_\lambda^{\alpha/2} f(t) & 0 & \cdots & 0 \end{pmatrix}^T, \quad
    \tilde{y}_{N, 0} = \begin{pmatrix}
        \bigotimes^1 A_\lambda^{\alpha/2} y_0 &
        \cdots &
        \bigotimes^N A_\lambda^{\alpha/2} y_0
    \end{pmatrix}^T.
\end{gather*}
The operators $\widetilde{B}_k$ and $\widetilde{F}_k(t)$ define the corresponding Kronecker sums of the operators $\widetilde{B} := A_\lambda^{\alpha/2} B (A_\lambda^{-\alpha/2} \otimes A_\lambda^{-\alpha/2})$ and $\tilde{f}(t) = A_\lambda^{\alpha/2} f(t)$ respectively. It is noted that the operators $A_k$ remain unchanged under this transformation as the operator $A_k$ commutes with the operators pre- and post-applied.
The operator $\widetilde{B}$ fulfills \eqref{assumption:BBilinearMapping} for $\alpha = 0$ with an adjusted constant, as can be seen by the estimate
\begin{align}
    \| \widetilde{B}\|_{\mathcal{L}(V_0^1(2), V^{-1 + \varepsilon})} &= \| A_\lambda^{\alpha/2} B (A_\lambda^{-\alpha/2} \otimes A_\lambda^{-\alpha/2})\|_{\mathcal{L}(V_0^1(2), V^{-1 + \varepsilon})} \nonumber \\
    &\le \| A^{\alpha/2}_\lambda\|_{\mathcal{L}(V^{\alpha - 1 + \varepsilon}, V^{-1 + \varepsilon})} \|A_\lambda^{-\alpha/2} \otimes A_\lambda^{-\alpha/2}\|_{\mathcal{L}(V_1^0(2), V_1^\alpha(2))} \| B \|_{\mathcal{L}(V_1^\alpha(2), V^{\alpha - 1 + \varepsilon})} \nonumber \\
    &\le c_{\lambda,\alpha} c_B(\alpha, \varepsilon), \label{eq:BLiftedConstant}
\end{align}
where $\alpha \in [0, \bar{\alpha} - 1]$ and the constant $c_{\lambda,\alpha}$ only depends on $A$, $L$, $\lambda$, and $\alpha$. If $\alpha = 0$, then $c_{\lambda,\alpha} = 1$.
This allows us to apply Lemma~\ref{lem:Bk} to $\widetilde{B}_k$, i.e., that $\widetilde{B}_k: V_1^0(k + 1) \to V_{-1 + \varepsilon}^0(k)$ is continuous. Moreover, we can use Lemma~\ref{lem:Fk} to show that $\widetilde{F}_k(t): V_1^0(k - 1) \to V_{-1 + \varepsilon}^0(k)$ is continuous provided that $f(t) \in V^\alpha$, which implies $A^{\alpha / 2}_\lambda f(t) \in H$. Although Lemmas~\ref{lem:Bk} and \ref{lem:Fk} are technically applied here only for the case $\alpha = 0$---specifically, for the transformed operator $\widetilde{B}$ and the transformed forcing, with constants adjusted to account for the lifting---we presented these results in the previous section for arbitrary $\alpha \geq 0$. This broader presentation highlights that the approach of \cite{Fursikov1989OnSystem,Fursikov1993MomentSide} can be readily adapted to our framework.

The operator $\widetilde{\mathcal{A}}_N(t)$ can be viewed as a weak perturbation of a $U_1^0(N)$-$Y(N)$ coercive operator. For that, we decompose the operator into
\begin{align*}
    \widetilde{\mathcal{A}}_N(t) = \mathcal{A}_{N, \text{diag}} + \widetilde{\mathcal{B}}_N + \widetilde{\mathcal{F}}_N(t),
\end{align*}
where $\mathcal{A}_{N, \text{diag}}$ is the diagonal part, $\widetilde{\mathcal{B}}_N$ the upper diagonal part, and $\widetilde{\mathcal{F}}_N(t)$ the lower diagonal part. The following propositions establish the foundation for demonstrating the weak perturbation property.
\begin{proposition}
\label{prop:AN}
Let $A$ fulfill assumptions \eqref{assumption:ABounded} and \eqref{assumption:AVHCoercive} with constants $\beta$, $\gamma$, and $\lambda$. Then, $\mathcal{A}_{N, \operatorname{diag}}$ is $U_1^0(N)$-$Y(N)$ coercive with the constants
\begin{align*}
\langle \mathcal{A}_{N, \operatorname{diag}} u, v \rangle_{U_1^0(N)} &\le \beta \|u\|_{U_1^0(N)} \|v\|_{U_1^0(N)},  \\
\langle \mathcal{A}_{N, \operatorname{diag}} v, v\rangle_{U_1^0(N)} + N \lambda \| v \|^2_{Y(N)} &\ge \gamma \| v \|_{U_1^0(N)}^2
\end{align*}
for all $u, v \in U_1^0(N)$.
\end{proposition}
\begin{proof}
Let $u = (u_1, \dots, u_N) \in U_1^0(N)$ and $v = (v_1, \dots, v_N) \in U_1^0(N)$ be arbitrary but fixed. Using Lemma~\ref{lem:Ak}, we can show that
\begin{align*}
    \langle \mathcal{A}_{N, \text{diag}} u, v\rangle_{U_1^0(N)} &= \sum_{k = 1}^N \langle A_k u_k, v_k\rangle_{V_1^0(k)} \le \sum_{k = 1}^N \beta \| u_k \|_{V_1^0(k)} \| v_k \|_{V_1^0(k)} \\
    &\le \beta \left(\sum_{k = 1}^N \| u_k \|_{V_1^0(k)}^2\right)^{\frac{1}{2}} \left(\sum_{k = 1}^N \| v_k \|_{V_1^0(k)}^2\right)^{\frac{1}{2}} = \beta \|u\|_{U_1^0(N)}\|v\|_{U_1^0(N)}.
\end{align*}
The coercivity property can be shown as follows:
\begin{align*}
    \langle \mathcal{A}_{N, \text{diag}} v, v\rangle_{U_1^0(N)} &= \sum_{k = 1}^N \langle A_k v_k, v_k\rangle_{V_1^0(k)} \ge \sum_{k = 1}^N \left( \gamma \|v_k \|_{V_1^0(k)}^2 - k \lambda \|v_k\|_{H(k)}^2 \right) \\
    &\ge \gamma \| v \|_{U_1^0(N)}^2 - N \lambda \|v\|_{Y(N)}^2. \qedhere
\end{align*}
\end{proof}
\begin{proposition}
\label{prop:BN}
Let $B$ fulfill assumption \eqref{assumption:BBilinearMapping} for some $\alpha \in [0, \bar{\alpha} - 1]$ and $\varepsilon \in [0, 1]$. Then, $\widetilde{\mathcal{B}}_N: U_1^0(N) \to U_{-1 + \varepsilon}^0(N)$ is continuous with
\begin{align*}
    \| \widetilde{\mathcal{B}}_N v\|_{U_{-1 + \varepsilon}^0(N)} \le c_{\mathcal{B}}(\alpha, \varepsilon) N^{\varepsilon / 2} \| v \|_{U_1^0(N)} \quad\text{for all } v \in U_1^0(N),
\end{align*}
where the constant $c_{\mathcal{B}}(\alpha, \varepsilon)$ can be bounded in terms of $c_B(\alpha, \varepsilon)$ and does not depend on $N$.
\end{proposition}
\begin{proof}
Let $v = (v_1, \dots, v_N) \in U_1^0(N)$ be arbitrary but fixed. Because of Lemma~\ref{lem:Bk} and the estimate \eqref{eq:BLiftedConstant}, it holds that
\begin{align*}
\| \widetilde{\mathcal{B}}_N v \|^2_{U_{-1 + \varepsilon}^0(N)} &= \sum_{k = 1}^{N - 1} \| \widetilde{B}_k v_{k + 1} \|^2_{V_{-1 + \varepsilon}^0(k)} \le 2 c_{\lambda,\alpha}^{2} c_B(\alpha, \varepsilon)^2 \sum_{k = 1}^{N - 1} k^{\varepsilon} \|v_{k + 1} \|_{V_1^0(k + 1)}^2 \\
&\le 2 c_{\lambda,\alpha}^{2} c_B(\alpha, \varepsilon)^2 N^{\varepsilon} \|v\|_{U_1^0(N)}^2.
\end{align*}
This implies the result with $c_{\mathcal{B}}(\alpha, \varepsilon) = \sqrt{2} c_{\lambda,\alpha} c_B(\alpha, \varepsilon)$.
\end{proof}
\begin{proposition}
\label{prop:FN}
Let $f(t) \in V^\alpha$ for some $\alpha \in [0, \bar{\alpha} - 1]$ and $\varepsilon \in [0, 1]$. Then, $\mathcal{\widetilde{F}}_N(t): U_1^0(N) \to U_{-1 + \varepsilon}^0(N)$ is continuous with
\begin{align*}
    \| \widetilde{\mathcal{F}}_N(t) v\|_{U_{-1 + \varepsilon}^0(N)} \le c_{\mathcal{F}}(\alpha, \varepsilon) \|f(t)\|_{V^\alpha} N^{\varepsilon / 2} \| v \|_{U_1^0(N)} \quad\text{for all } v \in U_1^0(N),
\end{align*}
where the constant $c_{\mathcal{F}}(\alpha, \varepsilon)$ does not depend on $f$ or $N$.
\end{proposition}
\begin{proof}
Let $v = (v_1, \dots, v_N) \in U_1^0(N)$ be arbitrary but fixed. With Lemma~\ref{lem:Fk}, we can show that
\begin{align*}
\| \widetilde{\mathcal{F}}_N(t) v \|^2_{U_{-1 + \varepsilon}^0(N)} &= \sum_{k = 2}^{N} \| \widetilde{F}_k(t) v_{k - 1} \|^2_{V_{-1 + \varepsilon}^0(k)} \le c_F(\varepsilon)^2 \|A_\lambda^{\alpha/2} f(t)\|_{H}^2 \sum_{k = 2}^{N} k^{\varepsilon} \|v_{k - 1} \|_{V_1^0(k - 1)}^2 \\
&\le \tilde{c}_{\lambda, \alpha}^{2} c_F(\varepsilon)^2 \|f(t)\|_{V^\alpha}^2 N^{\varepsilon} \|v\|_{U_1^0(N)}^2,
\end{align*}
where $\tilde{c}_{\lambda, \alpha} = \|A_\lambda^{\alpha/2} L^{-\alpha/2}\|_{\mathcal{L}(H, H)}$. This implies the result with $c_{\mathcal{F}}(\alpha, \varepsilon) = \tilde{c}_{\lambda,\alpha} c_F(\varepsilon)$
\end{proof}

The following lemma lets us identify solutions $\tilde{y}$ of \eqref{eq:TransformedCarleman} as solutions to \eqref{eq:TruncatedCarleman}.
\begin{lemma}
\label{lem:Isomorphism}
The linear mapping $\mathcal{D}_{\lambda, N}^{\xi}: W(0, T; U_1^\xi(N), U_{-1}^\xi(N)) \to W(0, T; U_1^0(N), U_{-1}^0(N))$ with $w \mapsto (t \mapsto \mathcal{D}_{\lambda, N}^{\xi} w(t))$
is an isomorphism with the inverse $\mathcal{D}_{\lambda, N}^{-\xi}$ for any $\xi \in \mathbb{R}$.
\end{lemma}

We are now in a position to show the well-posedness of equation \eqref{eq:TruncatedCarleman}.
\begin{theorem}[Well-posedness]
\label{thm:WellPosedness}
Let $T < \infty$. Moreover, let $A$ fulfill \eqref{assumption:ABounded} and \eqref{assumption:AVHCoercive}, $B$ fulfill \eqref{assumption:BBilinearMapping} for some $\alpha \in [0, \bar{\alpha} - 1]$ and $\varepsilon = 1$, $f \in L^\infty(0, T; V^\alpha)$, and $y_0 \in V^\alpha$. Then, the weak form of the truncated Carleman linearization \eqref{eq:TruncatedCarleman}, which reads as
\begin{align}
\begin{aligned}
    y_N'(t) + \mathcal{A}_N(t) y_N(t) &= f_N(t) \quad\text{in }L^2(0, T; \left(U_{1}^0(N)\right)'),\\
    y_N(0) &= y_{N, 0},
\end{aligned}
\tag{TCW}
\label{eq:WeakTruncatedCarleman}
\end{align}
has a unique solution $y_N \in W(0, T; U_{1}^\alpha(N), U_{-1}^\alpha(N))$ that depends continuously on $y_0$ and $f$.
\end{theorem}
\begin{proof}
Proposition~\ref{prop:AN} implies that $\mathcal{A}_{N, \text{diag}}$ is $U_1^0(N)$-$Y(N)$ coercive. From Proposition~\ref{prop:FN}, we know that for any $v \in U_1^0(N)$ the mapping $t \mapsto \widetilde{\mathcal{F}}_N(t) v$ fulfills
\begin{equation*}
\| t \mapsto \widetilde{\mathcal{F}}_N(t) v\|_{L^\infty(0, T; Y(N))} \le c_{\mathcal{F}}(\alpha, \varepsilon) \| f \|_{L^\infty(0, T; V^\alpha)} \sqrt{N} \| v \|_{U^0_1(N)}.
\end{equation*}
With Proposition~\ref{prop:BN}, it follows that $\widetilde{\mathcal{B}}_N + \widetilde{\mathcal{F}}_N(t)$ as a perturbation of $\mathcal{A}_{N, \text{diag}}$ fulfills the conditions in Lemma~\ref{lem:PerturbationLinear}. Thus, the weak form of \eqref{eq:TransformedCarleman}, i.e.,
\begin{align}
\begin{aligned}
    \tilde{y}_N'(t) + \widetilde{\mathcal{A}}_N(t) \tilde{y}_N(t) &= \tilde{f}_N(t) \quad \text{ in } L^2(0, T; \left(U_{1}^0(N)\right)'),\\
    \tilde{y}_N(0) &= \tilde{y}_{N, 0},
\end{aligned}
\tag{$\text{TCW}_\alpha$}
\label{eq:WeakTransformedCarleman}
\end{align}
has a unique solution $\tilde{y}_N \in W(0, T; U_1^0(N), U_{-1}^0(N))$. With Lemma~\ref{lem:Isomorphism} and $(\mathcal{D}_N^{-\alpha})^* U_1^0(N) \subseteq U_1^0(N)$, we can infer that $y_N = \mathcal{D}_{N}^{-\alpha} \tilde{y}_N \in W(0, T; U_1^\alpha(N), U_{-1}^\alpha(N))$ solves \eqref{eq:WeakTruncatedCarleman}. To show uniqueness, assume that $z_N \in W(0, T; U_1^\alpha(N), U_{-1}^\alpha(N))$ solves \eqref{eq:WeakTruncatedCarleman}. Then, $\tilde{z}_N(t) = \mathcal{D}_{N, \lambda}^{\alpha} z_N(t)$ solves \eqref{eq:WeakTransformedCarleman}, which implies that $\tilde{z}_N = \tilde{y}_N$. Lemma~\ref{lem:Isomorphism} lets us deduce that $z_N = y_N$.
\end{proof}

\section{Convergence in the Case of Small Nonlinearities}
\label{sec:Convergence}
Having established the existence and (partial) uniqueness of a solution to \eqref{eq:WeakTruncatedCarleman}, the question arises as to how well the linearization approximates the original nonlinear Cauchy problem. In particular, we are interested in whether a solution of the truncated Carleman linearization converges to the true solution of the nonlinear system for $N \to \infty$, i.e., whether the nonlinear dynamics can be recovered if the truncation level is chosen large enough.

Let $y_e \in W(0, T; V, V')$ be the solution to problem \eqref{eq:CauchyProblemWeak}, also referred to as the exact solution, and $y_N \in W(0, T; V^{1 + \alpha}, V^{-1 + \alpha})$ the solution to the truncated system \eqref{eq:WeakTruncatedCarleman}. The specific quantity we are interested in is the error defined as the norm of $\eta(t) := y_N^{(1)}(t) - y(t)$. For further analysis, we define the error of all moments $\eta_N(t) := y_N(t) - y_{N, e}(t)$, where $y_{N, e}(t)$ denotes the vector of the moments $y_e^{(k)}(t)$ for $k \in \{1, \dots, N\}$. The error function $\eta_N$ fulfills the dynamical system
\begin{equation}
\label{eq:ErrorDynamicalSystem}
\begin{aligned}
    \eta_N'(t) + \mathcal{A}_N(t) \eta_N(t) &= r_N(t), \\
    \eta_N(0) &= 0, \\
    r_N(t) &= \begin{pmatrix}
        0 & \cdots & 0 & -B_N y_e^{(N + 1)}(t)
    \end{pmatrix}^T.
\end{aligned}
\end{equation}
Under the assumptions of Theorem~\ref{thm:WellPosedness}, the operator $\widetilde{\mathcal{A}}_N(t)$ is uniformly $U_1^0(N)$-$Y(N)$ coercive with some constants $\gamma_N$ and $\lambda_N$, which enables us to apply the estimate \eqref{eq:ParabolicEstimate} from Lemma~\ref{lem:WellPosednessLinear}. This provides a bound of $\eta_N$ in terms of the right-hand side $r_N$, which depends on the exact solution of the original problem. However, the quality of the estimate depends on the coercivity constants.
The proof of Theorem~\ref{thm:WellPosedness} establishes the coercivity of the operator $\widetilde{\mathcal{A}}_N(t)$ by applying Lemma~\ref{lem:PerturbationLinear}. Careful consideration of the proof of Lemma~\ref{lem:PerturbationLinear} reveals that these constants depend on $N$; specifically, $\gamma_N \to 0$ and $\lambda_N \to \infty$ as $N \to \infty$. This indicates that the Cauchy problem becomes ill-posed in the asymptotic limit and that the estimate from Lemma~\ref{lem:WellPosednessLinear} deteriorates. In what follows, we will explore how we can obtain coercivity in a more controlled manner for the case of small nonlinearities and forcing, yielding coercivity constants that are more robust with respect to $N$. This provides a better understanding of the system's asymptotic behavior as $N \to \infty$.
\begin{proposition}
\label{prop:ImprovedCoercivity}
Let $A$ fulfill \eqref{assumption:ABounded} and \eqref{assumption:AVHCoercive} with constants $\beta$, $\gamma$, and $\lambda$, let $B$ fulfill \eqref{assumption:BBilinearMapping} for some $\alpha \in [0, \bar{\alpha} - 1]$ and $\varepsilon = 0$, and let $f \in L^\infty(0, T; V^\alpha)$. If
\begin{align*}
c_{\mathcal{P}} := c_{\mathcal{B}}(\alpha, 0) + c_{\mathcal{F}}(\alpha, 0) \|f\|_{L^\infty(0, T; V^\alpha)} < \gamma,
\end{align*}
then the operator $\widetilde{\mathcal{A}}_N(t)$ is bounded and coercive with constants
\begin{align*}
    \langle \widetilde{\mathcal{A}}_{N}(t) u, v \rangle_{U_1^0(N)} &\le \left(\beta + c_{\mathcal{P}}\right) \|u\|_{U_1^0(N)} \|v\|_{U_1^0(N)},  \\
    \langle \widetilde{\mathcal{A}}_{N}(t) v, v\rangle_{U_1^0(N)} + N \lambda \| v \|^2_{Y(N)} &\ge \left(\gamma - c_{\mathcal{P}}\right) \| v \|_{U_1^0(N)}^2,
\end{align*}
for almost every $t \in [0, T)$ and all $u, v \in U_1^0(k)$.
\end{proposition}
\begin{proof}
The boundedness and coercivity follow from Propositions~\ref{prop:AN}, \ref{prop:BN}, and \ref{prop:FN} for $\varepsilon = 0$. More specifically, let $u, v \in U_1^0(N)$ be arbitrary but fixed. Then, we obtain an upper bound through
\begin{align*}
&\langle \widetilde{\mathcal{A}}_{N}(t) u, v \rangle_{U_1^0(N)} = \langle\mathcal{A}_{N, \text{diag}}u, v\rangle_{U_1^0(N)} + \langle\widetilde{\mathcal{B}}_N u, v\rangle_{U_1^0(N)} + \langle\widetilde{\mathcal{F}}_N(t) u, v\rangle_{U_1^0(N)} \\
&\quad \le \beta \|u\|_{U_1^0(N)} \|v\|_{U_1^0(N)} + \| \widetilde{\mathcal{B}}_N u \|_{U_{-1}^0(N)} \|v\|_{U_1^0(N)} + \| \widetilde{\mathcal{F}}_N(t) u \|_{U_{-1}^0(N)} \|v\|_{U_1^0(N)} \\
&\quad \le \beta \|u\|_{U_1^0(N)} \|v\|_{U_1^0(N)} + c_{\mathcal{B}}(\alpha, 0) \|u\|_{U_1^0(N)} \|v\|_{U_1^0(N)} + c_{\mathcal{F}}(\alpha, 0) \|f(t)\|_{V^\alpha} \|u\|_{U_1^0(N)} \|v\|_{U_1^0(N)} \\
&\quad = \left(\beta + c_{\mathcal{B}}(\alpha, 0) + c_{\mathcal{F}}(\alpha, 0) \|f(t)\|_{V^\alpha} \right) \|u\|_{U_1^0(N)} \|v\|_{U_1^0(N)}
\end{align*}
for almost every $t \in [0, T)$. Similarly, the coercivity follows from
\begin{align*}
&\langle \widetilde{\mathcal{A}}_{N}(t) v, v \rangle_{U_1^0(N)} \ge \langle\mathcal{A}_{N, \text{diag}}v, v\rangle_{U_1^0(N)} - \left|\langle\widetilde{\mathcal{B}}_N v, v\rangle_{U_1^0(N)}\right| - \left|\langle\widetilde{\mathcal{F}}_N(t) v, v\rangle_{U_1^0(N)}\right| \\
&\quad \ge \gamma \|v\|_{U_1^0(N)}^2 - N \lambda \|v\|_{Y(N)}^2 - \| \widetilde{\mathcal{B}}_N v \|_{U_{-1}^0(N)} \|v\|_{U_1^0(N)} - \| \widetilde{\mathcal{F}}_N(t) v \|_{U_{-1}^0(N)} \|v\|_{U_1^0(N)} \\
&\quad \ge \left(\gamma - c_{\mathcal{B}}(\alpha, 0) - c_{\mathcal{F}}(\alpha, 0) \|f(t)\|_{V^\alpha}\right) \|v\|_{U_1^0(N)}^2 - N \lambda \|v\|_{Y(N)}^2. \qedhere
\end{align*}
\end{proof}

The estimate \eqref{eq:ParabolicEstimate} applied to \eqref{eq:ErrorDynamicalSystem} with the improved constants of Proposition~\ref{prop:ImprovedCoercivity} yields a bound of $\eta_N$ that can be expressed in terms of the $L^2(0, T; V^0_1(N))$-norm of $\bigotimes^{N + 1} y_e$. Thus, it is necessary to estimate the latter term.
\begin{proposition}
\label{prop:zNEstimate}
Let $T \in (0, \infty]$. For all $z \in W(0, T; V, V')$ and $N \in \mathbb{N}$ it holds that
\begin{align*}
    \| \bigotimes^N z \|_{L^2(0, T; V_1^0(N))} \le \sqrt{N} \left(\|z(0)\|_H +  \| z \|_{W(0, T; V, V')} \right)^N.
\end{align*}
\end{proposition}
\begin{proof}
Let $N \in \mathbb{N}$ and $z \in W(0, T; V, V')$ be arbitrary but fixed. 
With Lemma~\ref{lem:WBound}, it follows that
    \begin{align*}
&\| \bigotimes^{N} z \|_{L^2(0, T; V^0_1(N))}^2 = \int_0^T \| \bigotimes^{N} z(t) \|_{V_1^0(N)}^2 \mathrm{d} t = \int_0^T \sum_{i = 1}^N \prod_{j = 1}^N \| z(t) \|_{V^{\delta_{i, j}}}^2 \mathrm{d}t \\
&\quad= N \int_0^T \|z(t)\|_{V}^2 \| z(t) \|_H^{2 (N - 1)} \mathrm{d}t \le N (\|z\|_{L^\infty(0, T; H)})^{2 (N - 1)} \int_0^T \|z(t)\|_{V}^2 \mathrm{d}t \\
&\quad\le N \left(\|z(0)\|_H + \|z(t)\|_{W(0, T; V, V')} \right)^{2N}. \qedhere
\end{align*}
\end{proof}
% &\quad\le N \left(\|z(0)\|_H + \|z(t)\|_{W(0, T; V^{1 + \alpha}, V^{-1 + \alpha})} \right)^{2(N - 1)} \|z(t)\|_{W(0, T; V^{1 + \alpha}, V^{-1 + \alpha})} \\

Finally, we achieve convergence, provided the exact solution is small enough.
\begin{theorem}[Convergence]
\label{thm:Convergence}
Let $T < \infty$. Let $A$, $B$, and $f$ fulfill all assumptions of Proposition~\ref{prop:ImprovedCoercivity} and $c_{\mathcal{P}} < \gamma$. The initial condition is assumed to fulfill $y_0 \in V^\alpha$. Moreover, assume that \eqref{eq:CauchyProblemWeak} admits a solution $y_e \in W(0, T; V^{1 + \alpha}, V^{-1 + \alpha})$. Then, it holds that
\begin{align*}
    &\| \eta \|_{L^\infty(0, T; V^\alpha)} + \|\eta\|_{W(0, T; V^{1 + \alpha}, V^{-1 + \alpha})} \\
    &\quad\le c_1 \sqrt{N + 1} \left(c_2 \exp(\lambda T)\left( \|y_0\|_{V^\alpha} + \| y_e\|_{W(0, T; V^{1 + \alpha}, V^{-1 + \alpha})}\right) \right)^{N + 1},
\end{align*}
where $\eta(t) := y_N^{(1)}(t) - y_e(t)$ and $y_N \in W(0, T; U_1^\alpha(N), U_{-1}^\alpha(N))$ is the solution to \eqref{eq:WeakTruncatedCarleman}. The constants $c_1, c_2 > 0$ do not depend on $N$, $T$, $y_0$, or $y_e$. If $\lambda = 0$, one can choose $T = \infty$.
\end{theorem}
\begin{proof}
Due to Proposition~\ref{prop:ImprovedCoercivity}, we know that \eqref{eq:WeakTruncatedCarleman} has a unique solution $y_N \in \allowbreak W(0, T; \allowbreak U_1^0(N),\allowbreak  U_{-1}^0(N))$. Based on the definition $\eta_N(t) = y_N(t) - y_{N, e}(t)$, we define $\tilde{\eta}_N(t) = \mathcal{D}_{\lambda, N}^\alpha \eta_N(t)$, which satisfies the transformed dynamical system
\begin{equation*}
\begin{aligned}
    \tilde{\eta}_N'(t) + \widetilde{\mathcal{A}}_N(t) \tilde{\eta}_N(t) &= \tilde{r}_N(t) \quad \text{in } L^2(0, T; \left(U_1^0(N)\right)'), \\
    \tilde{\eta}_N(0) &= 0, \\
    \tilde{r}_N(t) &= \begin{pmatrix}
        0 & \cdots & 0 & -\widetilde{B}_N \tilde{y}_e^{(N + 1)}(t)
    \end{pmatrix}^T
\end{aligned}
\end{equation*}
with $\tilde{y}_e^{(N + 1)}(t) := \bigotimes^{N + 1} A_{\lambda}^{\alpha/2} y_e(t)$. From Lemma~\ref{lem:WellPosednessLinear} and Proposition~\ref{prop:ImprovedCoercivity}, we know that the solution to this dynamical system fulfills the estimate
\begin{align}
    &\| \tilde{\eta}_N \|_{L^\infty(0, T; Y(N))}^2 + \left(\gamma - c_{\mathcal{P}}\right)\|\tilde{\eta}_N\|_{L^2(0, T; U_1^0(N))}^2 \nonumber \\
    &\quad\le \frac{1}{\gamma - c_{\mathcal{P}}} \int_0^T \exp(2 N \lambda (T - t)) \|\tilde{r}_N(t)\|^2_{U_{-1}^0(N)} \mathrm{d}t \nonumber \\
    &\quad\le \frac{1}{\gamma - c_{\mathcal{P}}} \exp(2 N \lambda T) \| \tilde{r}_N \|_{L^2(0, T; U_{-1}^0(N))}^2.
\label{eq:ProofConvergenceBoundByResidual}
\end{align}
Lemma~\ref{lem:Bk} and Proposition~\ref{prop:zNEstimate} imply that
\begin{align}
    \|\tilde{r}_N\|_{L^2(0, T; U_{-1}^0(N))} &= \| \widetilde{B}_N \tilde{y}_e^{(N + 1)} \|_{L^2(0, T; V_{-1}^0(N))} \le \sqrt{2} c_{\lambda,\alpha} c_B(\alpha, 0) \| \tilde{y}_e^{(N + 1)} \|_{L^2(0, T; V_1^0(N + 1))} \nonumber \\
    &\le \sqrt{2} c_{\lambda,\alpha} c_B(\alpha, 0) \sqrt{N + 1} \left(\|\tilde{y}_e(0)\|_{H} + \| \tilde{y}_e \|_{W(0, T; V, V')} \right)^{N + 1}.
\label{eq:ProofConvergenceResidual}
\end{align}
Moreover, we can relate the estimate of $\eta_N'$ with $\eta_N$ by
\begin{align}
\|\tilde{\eta}_N'\|_{L^2(0, T; U_{-1}^0(N))}^2 &= \int_0^T \| -\widetilde{\mathcal{A}}_N(t) \tilde{\eta}_N(t) + \tilde{r}_N(t) \|_{U_{-1}^0(N)}^2 \mathrm{d}t \nonumber \\
&\le \int_0^T \left((\beta + c_{\mathcal{P}}) \|\tilde{\eta}_N(t)\|_{U_1^0(N)}^2 + \|\tilde{r}_N(t)\|_{U_{-1}^0(N)}^2 \right) \mathrm{d}t \nonumber \\
&= (\beta + c_{\mathcal{P}}) \| \tilde{\eta}_N\|_{L^2(0, T; U_1^0(N))}^2 + \|\tilde{r}_N\|_{L^2(0, T; U_{-1}^0(N))}^2.
\label{eq:ProofConvergenceTimeDerivative}
\end{align}
By simplifying the norms and transformed variables, this, finally, leads to the estimate
\begin{alignat*}{3}
    &\| \eta\|_{L^\infty(0, T; V^\alpha)} + \|\eta\|_{W(0, T; V^{1 + \alpha}, V^{-1 + \alpha})} = \tilde{c}_{\lambda, \alpha} \left(\| \tilde{\eta}\|_{L^\infty(0, T; H)} + \|\tilde{\eta}\|_{W(0, T; V, V')}\right) \span\span \\
    &\quad&\le\;\;\;&\tilde{c}_{\lambda, \alpha}\left(\|\tilde{\eta}_N\|_{L^\infty(0, T; Y(N))} + \|\tilde{\eta}_N\|_{W(0, T; U_{1}^0(N), U_{-1}^0(N))}\right) \\
    &\quad&\ule{\eqref{eq:ProofConvergenceTimeDerivative},\eqref{eq:ProofConvergenceBoundByResidual}}\;\;\;& \tilde{c}_{\lambda, \alpha}\bar{c}(\gamma, \beta, c_{\mathcal{P}}) \exp(N \lambda T) \|\tilde{r}_N\|_{L^2(0, T; U_{-1}^0(N))}  \\
    &\quad&\ule{\eqref{eq:ProofConvergenceResidual}}\;\;\;& \tilde{c}_{\lambda, \alpha}\bar{c}(\gamma, \beta, c_{\mathcal{P}}) \sqrt{2} c_{\lambda,\alpha} c_B(\alpha, 0) \exp(N \lambda T) \sqrt{N + 1} \left(\| \tilde{y}_0\|_{H} + \| \tilde{y}_e\|_{W(0, T; V, V')} \right)^{N + 1} \\
    &\quad &=\;\;\;& c_1 \sqrt{N + 1} \left(\exp(\lambda T) \left( \| \tilde{y}_0\|_{H} + \| \tilde{y}_e\|_{W(0, T; V, V')}\right) \right)^{N + 1} \\
    &\quad &\le\;\;\;& c_1 \sqrt{N + 1} \left(c_2 \exp(\lambda T) \left( \| y_0\|_{V^\alpha} + \| y_e\|_{W(0, T; V^{1 + \alpha}, V^{-1 + \alpha})}\right) \right)^{N + 1},
\end{alignat*}
where $\tilde{c}_{\lambda, \alpha}$ is a constant depending only on $\|L^{\alpha / 2} \allowbreak A_\lambda^{-\alpha/2}\|_{\mathcal{L}(H, H)}$, $\|L^{\alpha / 2} \allowbreak A_\lambda^{-\alpha/2}\|_{\mathcal{L}(V, V)}$, and $\|L^{\alpha / 2} \allowbreak A_\lambda^{-\alpha/2}\|_{\mathcal{L}(V', V')}$, while $c_2$ is defined analogously, but with $A_\lambda^{\alpha/2} L^{-\alpha / 2}$. The constant $\bar{c}(\gamma, \beta, c_{\mathcal{P}}) > 0$ depends only on the parameters listed. In the special case $\alpha = 0$, we have $\tilde{c}_{\lambda, \alpha} = c_2 = 1$.
\end{proof}

\begin{remark}
Theorem~\ref{thm:Convergence} assumes that the exact solution $y_e$ exhibits increased regularity compared to the results discussed in Section~\ref{sec:NonlinearCauchyProblem}. Proving improved regularity can be achieved with additional assumptions on the initial condition, the forcing, and the involved operators. In the linear case, this is broadly covered; see, e.g., \cite[pp.~180 ff.]{Bensoussan2007RepresentationSystems} for general weak forms of Cauchy problems on Hilbert spaces and \cite[pp.~410 ff.]{Evans2010PartialEquations} specifically for PDEs. In the nonlinear case, this has also been explored, particularly through concepts like strong variational solutions to the three-dimensional Navier--Stokes equations; see, e.g., \cite{Barbu2011StabilizationFlows}. Alternatively, one could extend assumption~\eqref{assumption:BBounded} in such a way that it accounts for $\alpha$ similarly to \eqref{assumption:BBilinearMapping}. This would allow for an equivalent of Lemma~\ref{lem:WellPosednessNonlinear} with improved regularity. However, a detailed discussion is deferred as it would exceed the scope of this work.
\end{remark}

Theorem~\ref{thm:Convergence} can be refined if we assume $\alpha = 0$ and leverage the results about solutions to the nonlinear Cauchy problem~\eqref{eq:CauchyProblemWeak}.
\begin{corollary}[Convergence for $\alpha = 0$]
\label{corollary:ConergenceAlphaZero}
    Let $T < \infty$. Let $A$ fulfill \eqref{assumption:ABounded} and \eqref{assumption:AVHCoercive} with constants $\beta$, $\gamma$, and $\lambda$, let $B$ fulfill \eqref{assumption:BBounded}, and let $f \in L^2(0, T; V') \cap L^\infty(0, T; H)$ and $y_0 \in H$. Additionally, assume that $B$, $f$, and $y_0$ are so small that $c_{\mathcal{P}} < \gamma$ and $\| y_0 \|_H + \|f\|_{L^2(0, T; V')} \le \rho(T)$, where $c_{\mathcal{P}}$ and $\rho(T)$ are the constants from Proposition~\ref{prop:ImprovedCoercivity} and Lemma~\ref{lem:WellPosednessNonlinear}. Let $y_e \in W(0, T; V, V')$ be the unique solution to \eqref{eq:CauchyProblemWeak}. Then, it holds that
    \begin{align*}
        \| \eta \|_{L^\infty(0, T; H)} + \|\eta\|_{W(0, T; V, V')} \le c_1\sqrt{N + 1} \left( 2 c_N \exp(2 \lambda T) \left(\|y_0\|_H + \|f\|_{L^2(0, T; V')}\right)\right)^{N + 1},
    \end{align*}
    where $\eta(t) = y_N^{(1)} - y_e(t)$ and $y_N \in W(0, T; U_1^0(N), U_{-1}^0(N))$ is the unique solution to \eqref{eq:WeakTruncatedCarleman} for any $N \in \mathbb{N}$. The constant $c_1 > 0$ does not depend on $N$, $T$, $y_0$, or $f$, and $c_N$ is the constant from Lemma~\ref{lem:WellPosednessNonlinear}. If $\lambda = 0$, one can choose $T = \infty$.
\end{corollary}
\begin{proof}
Since assumption \eqref{assumption:BBounded} implies \eqref{assumption:BBilinearMapping} for $\alpha = 0$, the statement follows from Theorem~\ref{thm:Convergence} and Lemma~\ref{lem:WellPosednessNonlinear}.
\end{proof}

\subsection{Assumptions and Implications of the Result}

Subsequently, we interpret the assumptions of Theorem~\ref{thm:Convergence} and Corollary~\ref{corollary:ConergenceAlphaZero}, and highlight their influence on the linearization's performance.

First, consider the assumptions $c_{\mathcal{P}} < \gamma$ and $\|y_0\|_H + \|f\|_{L^2(0, T; V')} \le \rho(T)$. The nonlinear operator $B$ is assumed to fulfill \eqref{assumption:BBilinearMapping}. The constant $c_{\mathcal{B}}(\alpha, 0)$ is bounded by $c_{B}(\alpha, \varepsilon)$, in particular, there is a $\mu > 0$ such that $c_{\mathcal{B}}(\alpha, \varepsilon) \le \mu c_B(\alpha, \varepsilon)$. The constant~$c_{B}(\alpha, \varepsilon)$, in turn, determines the size of the nonlinearity $B$. This becomes apparent if we assume that $B$ is given as a scaled nonlinear term $B_0$, i.e., $B = \omega B_0$ with $\omega \in \mathbb{R}$. Then, one can choose $c_{B}(\alpha, \varepsilon) = \omega c_{B_0}(\alpha, \varepsilon)$. For example, in fluid flow problems, $B$ might represent a convection term, and $\omega$ the Reynolds number. What this means for Theorem~\ref{thm:Convergence} is that $\omega$ and $f$ have to be chosen sufficiently small that $c_{\mathcal{P}} < \gamma$, i.e., the Carleman linearization is only admissible for small nonlinearities and forcing. The coercivity constant $\gamma$ can be interpreted as a measure of the stability of the operator $A$. Hence, more stable linear parts allow for larger nonlinearities. Similar implications hold for the assumption $\|y_0\|_H + \|f\|_{L^2(0, T; V')} \le \rho(T)$ in Corollary~\ref{corollary:ConergenceAlphaZero}. The constant $\rho(T)$ is inversely proportional to $c_B(\alpha, 0)$, and, by extension, to $\omega$. This means that the set of admissible $y_0$ and $f$ restricted by this assumption grows as the nonlinearity gets smaller. Moreover, if $\lambda > 0$, i.e., if the linear part is asymptotically unstable, it holds that $\rho(T) \to 0$ exponentially in $T$. So, the convergence result can only be applied on finite time horizons if $\lambda > 0$.

Turning to the error bound of Theorem~\ref{thm:Convergence}, we observe that convergence of the approximate solution is guaranteed if and only if
\begin{align*}
    c_2 \exp(\lambda T)\left( \|y_0\|_{V^\alpha} + \| y_e\|_{W(0, T; V^{1 + \alpha}, V^{-1 + \alpha})}\right) < 1.
\end{align*}
Since $\lambda \ge 0$, this necessitates for the solution to fulfill $c_2 (\|y_0\|_{V^\alpha} + \| y_e\|_{W(0, T; V^{1 + \alpha}, V^{-1 + \alpha})}) < 1$. On the one hand, this enforces an upper bound on the initial condition. On the other hand, it restricts the size of the exact solution, which entails multiple restrictions on our dynamical system. Since $\| y_e\|_{W(0, T; V^{1 + \alpha}, V^{-1 + \alpha})} \to 0$ as $T \to 0$, the condition might give an upper bound on feasible $T$ and larger time horizons also lead to poorer approximation properties of the linearization. This part of the estimate, however, does not necessarily lead to an upper bound of feasible $T$ in general since the solution $y_e$ can be bounded on the unbounded time horizon $[0, \infty)$ in the case $\lambda = 0$. Additionally, it is suggested that larger nonlinearities may lead to solutions $y_e$ with larger norms, resulting in poorer convergence or leaving one without a guarantee of convergence at all. Taking into account how Corollary~\ref{corollary:ConergenceAlphaZero} relates the norm of the solution to the forcing, the condition for convergence reads as
\begin{align*}
c_2 \exp(\lambda T) \left(\|y_0\|_H + \|f\|_{L^2(0, T; V')}\right) < 1
\end{align*}
and, therefore, also implies an upper bound on $f$. In the unstable case $\lambda > 0$, we have that convergence is only guaranteed for certain bounded time horizons with fixed $y_0$ and $f$.

These observations coincide with what is generally known about the linearization in the finite-dimensional case. Theorem~\ref{thm:Convergence} can be seen as an equivalent result to Theorem~4.2 in \cite{Forets2017ExplicitLinearization} without forcing and Corollary~\ref{corollary:ConergenceAlphaZero} to Theorems~3.2, 3.5, and 3.6 in \cite{Amini2025CarlemanApproximations} including the case of forcing. A notable difference is that our estimates include an additional factor $\sqrt{N + 1}$, which means that our results only show (sub-)exponential convergence.

\begin{remark}
The result from Theorem~\ref{thm:Convergence} and Corollary~\ref{corollary:ConergenceAlphaZero} is designed for the truncated Carleman linearization, i.e., $y_{N + 1} = 0$ is assumed in the chain of moment equations. As we have seen, there is an upper bound for admissible norms of $\|y_0\|$, and the linearization does not converge even locally in time if the initial condition is too large. However, if an estimate $\hat{y}_{N + 1}(t) = \bigotimes^{N + 1} \hat{y}(t) \approx y_{N + 1}(t)$ is available, we can refine the linearization and the corresponding error bound. This estimate can be incorporated into the linearization by changing the right-hand side to $\hat{f}_N(t) = (f(t), 0, \dots, 0, -B_N \hat{y}(t))^T$. By adapting Theorem~\ref{thm:Convergence} for this case, we end up with
\begin{align*}
    &\| \eta \|_{L^\infty(0, T; V^\alpha)} + \|\eta\|_{W(0, T; V^{1 + \alpha}, V^{-1 + \alpha})} \\
    &\quad\le c_1 \sqrt{N + 1} \left(c_2 \exp(\lambda T)\left( \|y_0 - \hat{y}(0)\|_{V^\alpha} + \| y_e - \hat{y}\|_{W(0, T; V^{1 + \alpha}, V^{-1 + \alpha})}\right) \right)^{N + 1}.
\end{align*}
Hence, if a guess $\hat{y}$ is close enough to $y_e$, we achieve convergence. Furthermore, this implies that if the assumptions of Theorem~\ref{thm:Convergence} are fulfilled and $\hat{y}(0) = y_0$, then there exists a (sufficiently small) $T > 0$ such that the linearization converges. This provides a local-in-time convergence result for arbitrarily large data $y_0$ and $f$.
\end{remark}

\subsection{Examples of Operators $A$ and $B$}
We examine specific cases of PDE problems where the assumptions~\eqref{assumption:BBounded} and \eqref{assumption:BBilinearMapping} on $B$ are fulfilled. We assume that $A$ is a second-order elliptic differential operator on a bounded Lipschitz domain $\Omega \in \mathbb{R}^d$ for $d \in \mathbb{N}$. We restrict ourselves to the case of homogeneous Dirichlet boundary conditions and, therefore, choose $D(A) = H^2(\Omega) \cap H^1_0(\Omega)$ and $H = L^2(\Omega)$, or vectorized versions of them. In this context, $L = -\Delta + I$ with $\bar{\alpha} = 2$ serves as a suitable choice, and $V = H^1_0(\Omega)$. Consequently, the $V$-norm is equivalent to the standard $H^1(\Omega)$-norm.

\subsubsection{Zeroth-order derivatives}
First, we consider the case of $B$ being a zeroth-order quadratic term, i.e., of the form
\begin{align*}
    B(u, v)(x) := \psi(x) u(x) v(x)
\end{align*}
for some $\psi \in L^\infty(\Omega)$. From Ladyzhenskaya's inequality, we know that $H^1(\Omega)$ is continuously embedded in $L^4(\Omega)$ for dimensions $d \le 2$. So, there is a constant $c_{LA}$ such that
\begin{align*}
    \| u \|_{L^4(\Omega)} \le c_{LA} \| u \|_{L^2(\Omega)}^{\frac{1}{2}} \| u \|_{H^1(\Omega)}^{\frac{1}{2}}
\end{align*}
for all $u \in H^1(\Omega)$. Thus, for $d \le 2$
\begin{align*}
    \left|\langle B(u, v), w \rangle_V\right| &\le \|\psi\|_{L^\infty(\Omega)}\|u\|_{L^4(\Omega)} \|v\|_{L^4(\Omega)} \| w \|_{L^2(\Omega)} \\
    &\le c_{LA}^2 \|\psi\|_{L^\infty(\Omega)} \| u \|_{L^2(\Omega)}^{\frac{1}{2}} \| u \|_{H^1(\Omega)}^{\frac{1}{2}} \| v \|_{L^2(\Omega)}^{\frac{1}{2}} \| v \|_{H^1(\Omega)}^{\frac{1}{2}} \|w \|_{L^2(\Omega)}
\end{align*}
for all $u, v \in V$ and $w \in H$.
This proves the required bound for \eqref{assumption:BBounded} and \eqref{assumption:BBilinearMapping} for $\alpha = 0$ and $\varepsilon \in [0, 1]$. In the case $d = 3$, we can prove \eqref{assumption:BBounded}. Using the Gagliardo--Nirenberg interpolation inequality in bounded domains
\begin{align*}
    \| u \|_{L^3(\Omega)} \le c_{GN} \|u\|_{L^2(\Omega)}^{\frac{1}{2}} \|u\|_{H^1(\Omega)}^{\frac{1}{2}},
\end{align*}
we obtain the estimate
\begin{align*}
    \left|\langle B(u, v), w \rangle_V\right| &\le \|\psi\|_{L^\infty(\Omega)}\|u\|_{L^3(\Omega)} \|v\|_{L^3(\Omega)} \| w \|_{L^3(\Omega)} \\
    &\le c_{GN}^3 \|\psi\|_{L^\infty(\Omega)} \| u \|_{L^2(\Omega)}^{\frac{1}{2}} \| u \|_{H^1(\Omega)}^{\frac{1}{2}} \| v \|_{L^2(\Omega)}^{\frac{1}{2}} \| v \|_{H^1(\Omega)}^{\frac{1}{2}} \|w \|_{H^1(\Omega)}
\end{align*}
for all $u, v, w \in V$. Thus, assumption \eqref{assumption:BBounded} is fulfilled for $d = 3$, and, therefore, also \eqref{assumption:BBilinearMapping} with $\alpha = 0$ and $\varepsilon = 0$; however, it does not hold for $\varepsilon = 1$ in general.

\subsubsection{Fluid flow problems}
\label{sec:FlowProblems}
In the case of the incompressible Navier--Stokes equation, the original coupled problem can be rephrased as a parabolic Cauchy problem with a quadratic nonlinearity by the use of the Leray projection $P$ and by restricting $H$ to solenoidal vector fields (see, e.g., \cite{Barbu2011StabilizationFlows}), and, therefore, fits our framework. Then, the nonlinear term reads as
\begin{align*}
    B(u, v) = \frac{1}{2}\left(\widehat{B}(u, v) + \widehat{B}(v, u)\right), \quad \widehat{B}(u, v) = P(u \cdot \nabla v).
\end{align*}
In this case, assumption \eqref{assumption:BBilinearMapping} holds for $\alpha > d/2 - 1$ and $0 \le \varepsilon < \alpha - d/2 + 1$, cf.~\cite{Fursikov1993MomentSide}. In the case of $d = 1,2$, assumption \eqref{assumption:BBounded} is fulfilled; see, e.g., \cite{Breiten2019FeedbackApproximation,Temam1979Navier-StokesAnalysis}. While this implies \eqref{assumption:BBilinearMapping} for $\alpha = 0$ and $\varepsilon = 0$ for $d = 1, 2$, it does not guarantee this property for $\varepsilon = 1$. This means that Theorem~\ref{thm:WellPosedness} cannot be applied to show well-posedness, but one has to use Theorem~\ref{thm:Convergence} coupled with assumptions of small nonlinearities, forcings, and initial conditions. This example underscores the importance of carefully examining the various assumptions on $B$.

% Furthermore, the assumptions may be shown for more specific models as well. For example, the nonlinear term of the shallow water equations is given by
% \begin{align*}
%     B((u, h), (v, g)) &= \frac{1}{2} \left( \widehat{B}((u, h), (v, g)) + \widehat{B}((v, g)), (u, h) \right), \\
%     \widehat{B}((u, h), (v, g)) &= \left( \nabla \cdot (h v), u \cdot \nabla v \right)^T.
% \end{align*}
% By applying the divergence theorem and incorporating the homogeneous boundary conditions, it can be shown that this fulfills \eqref{assumption:BBounded}, as well as \eqref{assumption:BBilinearMapping} for $\alpha = 0$ and $\varepsilon = 0$ in dimensions $d = 1, 2$.

\subsubsection{Nonlocal interaction terms}
\label{sec:InteractionTerm}
Some problems modelled by PDEs can involve nonlocal and possibly nonlinear terms. Examples of such problems can be found in multiscale particle dynamics. Our framework also covers nonlocal nonlinear problems including low-order derivatives. For example, the so-called dynamic density functional theory \cite{Marconi2000DynamicFluids,Chan2005Time-DependentFluids} tackles particle dynamics problems by approximating them by a PDE with a nonlocal term of the form
\begin{align*}
    B(u, v) = \frac{1}{2}\left( \widehat{B}(u, v) + \widehat{B}(v, u) \right), \quad \widehat{B}(u, v)(x) = \nabla_x \cdot \int_\Omega u(x) v(x') K(x, x') \mathrm{d}x'
\end{align*}
for some vector-valued kernel function $K(x, x')$. If we assume that $K \in (L^\infty(\Omega \times \Omega))^d$ and $\nabla_x \cdot K \in L^\infty(\Omega \times \Omega)$, we can show
{\allowdisplaybreaks
\begin{align*}
    \left|\langle \widehat{B}(u, v), w\rangle_V\right| &= \left|\int_\Omega w(x) \nabla_x \cdot \int_\Omega u(x) v(x') K(x, x') \mathrm{d}x' \mathrm{d}x \right| \\
    &= \left| \int_\Omega \int_\Omega \left(\nabla_x w(x)\right) u(x) v(x') K(x, x') \mathrm{d}x' \mathrm{d}x \right| \\
    &\le \| K \|_{L^\infty(\Omega \times \Omega)} \| u \|_{L^2(\Omega)} \| v \|_{L^1(\Omega)} \|w \|_{H^1(\Omega)}
\end{align*}
}%
by using integration by parts combined with the homogeneous boundary conditions. So, $B$ fulfills \eqref{assumption:BBounded} and, therefore, also \eqref{assumption:BBilinearMapping} for $\alpha = 0$ and $\varepsilon = 0$. Furthermore, we obtain that
\begin{align*}
    &\left| \langle \widehat{B}(u, v), w\rangle_V\right| = \left| \int_\Omega \int_\Omega w(x) v(x') \left[ u(x) \nabla \cdot K(x, x') + \left(\nabla u(x)\right) \cdot K(x, x') \right] \mathrm{d}x' \mathrm{d}x \right| \\
    &\quad\le \int_\Omega \int_\Omega \left[\left|w(x) v(x') u(x) \nabla \cdot K(x, x')\right| + \left|w(x) v(x') \left(\nabla u(x)\right) \cdot K(x, x') \right| \right] \mathrm{d}x' \mathrm{d}x \\
    &\quad\le \left( \| \nabla \cdot K \|_{L^\infty(\Omega \times \Omega)} + \|K \|_{(L^\infty(\Omega \times \Omega))^d} \right) \left( \|u\|_{L^2(\Omega)} \|v\|_{H^1(\Omega)} + \|u\|_{H^1(\Omega)} \|v\|_{L^2(\Omega)} \right) \|w\|_{L^2(\Omega)},
\end{align*}
which implies \eqref{assumption:BBilinearMapping} for $\alpha = 0$ and $\varepsilon = 1$. It is noted that this holds true for any dimension $d \in \mathbb{N}$.

\section{Numerical Approximations}
\label{sec:NumericalMethods}
Since we consider infinite-dimensional Hilbert spaces $H$, the presented parabolic Cauchy problems stemming from the Carleman linearization cannot be solved exactly in general. To address this, one must first discretize the equations and compute a numerical approximation. As we will see, the convergence analysis in Section~\ref{sec:Convergence} not only improves the understanding of the Carleman linearization for infinite-dimensional problems in a continuous setting but also clarifies its discretized counterpart. By pursuing a so-called linearize-then-discretize approach, the error analysis of the truncation is decoupled from the discretization error. This substantially differentiates our approach from traditional \textit{discretize-then-linearize} methods, which analyze the Carleman linearization of a finite-dimensional approximation of the original Cauchy problem. Additionally, the linearize-then-discretize approach will enable alternative discretization approaches.

As shown before, the truncated Carleman linearization results in a parabolic problem. The numerical solution of parabolic problems is broadly covered in the literature; see, e.g., \cite{Zeidler1990II/Operators,Thomee2006GalerkinProblems}. Many discretization techniques can be organized into two separate topics: the spatial and the temporal discretization.
In this section, we focus only on the spatial discretization, leaving aside the discussion of temporal discretizations as they are not relevant to this paper. Specifically, we analyze the so-called semi-discrete problem, which is obtained by only approximating $H$ but keeping the problem continuous in time. Subsequently, we outline the fundamental implications for the semi-discrete Cauchy problem based on the obtained results on the Carleman linearization.

We choose to use a conforming Galerkin method for the spatial discretization; see, e.g., \cite{Thomee2006GalerkinProblems}. It is assumed that all assumptions of Corollary~\ref{corollary:ConergenceAlphaZero} are fulfilled. Let $U_h(N) \subset U_1^0(N)$ be a family of finite-dimensional subspaces of $U_1^0(N)$ with the property
\begin{align*}
\label{assumption:GalerkinApproximationProperty}
\refstepcounter{assumptionCounter}
\tag{A\theassumptionCounter}
\inf_{u_h \in U_h(N)}\| u_h - u\|_{Y(N)} \to 0\quad\text{as} \quad h \to 0 \quad \text{for all } u \in U_1^0(N).
\end{align*}
In the case of Galerkin finite element methods for PDEs, $h$ denotes the maximum cell size, for instance. We then seek a function $y_{N, h} \in W(0, T; U_h(N), (U_h(N))')$ such that \eqref{eq:WeakTruncatedCarleman} is fulfilled, in which we only test against functions in $U_h(N)$. This equation is referred to as the semi-discretization of \eqref{eq:WeakTruncatedCarleman}.
The (semi-)discretization leads to an additional error source in our linearization. This becomes apparent when considering
\begin{align*}
\| y_{N, h}^{(1)} - y_e \|_{L^\infty(0, T; H)} &\le \| y_{N}^{(1)} - y_e \|_{L^\infty(0, T; H)} + \|y_{N, h}^{(1)} - y_{N}^{(1)}\|_{L^\infty(0, T; H)} \\
&\le \| y_{N}^{(1)} - y_e \|_{L^\infty(0, T; H)} + \|y_{N, h} - y_{N}\|_{L^\infty(0, T; Y(N))}.
\end{align*}
The first term of the right-hand side, the linearization error, is analyzed in Corollary~\ref{corollary:ConergenceAlphaZero}, while the second term, the discretization error, purely measures the approximation error of the Galerkin discretization. We thereby established a separation of the two error sources. The same observation remains valid if the norms $L^\infty(0, T; H)$ and $L^\infty(0, T; Y(N))$ are replaced by the norms $L^2(0, T; V)$ and $L^2(0, T; U_1^0(N))$, respectively.
The following lemma is a variation of Theorem 23.A in \cite{Zeidler1990II/Operators} and can be found in various standard references on the topic.
\begin{lemma}
Let all assumptions of Corollary~\ref{corollary:ConergenceAlphaZero} be fulfilled. Moreover, let $N \in \mathbb{N}$ be fixed and $U_h(N) \subset U_1^0(N)$ be a family of subspaces fulfilling \eqref{assumption:GalerkinApproximationProperty}. Let $y_N \in W(0, T; U_1^0(N), U_{-1}^0(N))$ be the solution to \eqref{eq:WeakTruncatedCarleman} and $y_{N, h} \in W(0, T; U_h(N), (U_h(N))')$ its Galerkin approximation. Then, it holds that
\begin{align*}
    \| y_N(t) - y_{N, h}(t)\|_{L^\infty(0, T; Y(N))} &\to 0, \\
    \|y_N - y_{N, h}\|_{L^2(0, T; U_1^0(N))} &\to 0,
\end{align*}
as $h \to 0$.
\end{lemma}
Hence, the assumptions ensure that the discretization error vanishes as $h \to 0$ for a fixed $N$. Informally, this means that the discretized Carleman linearization converges to the exact solution as $N \to \infty$ and $h \to 0$ simultaneously. However, careful selection of $h$ is crucial for ensuring overall convergence behavior. As $N$ increases, the dimensionality of problem \eqref{eq:WeakTruncatedCarleman} grows, possibly detoriating the discretization's approximation properties due to the curse of dimensionality. Although these factors strongly depend on the specific discretization method, this suggests that one has to choose $h$ such that it decreases sufficiently fast as a function of $N$ to gain overall convergence.

It remains an open question how one can choose $U_h(N)$ appropriately. The subsequent sections present a standard choice of discretization but also initiate the idea for nonstandard approaches that exploit the structure of the linearization.

\subsection{Standard Discretization}
We begin by considering a discretization that intuitively mirrors the construction of the spaces $V_1^0(k)$ in their discrete form. Let $V_h \subset V$ be a family of subspaces that fulfills
\begin{align*}
    \inf_{v_h \in V_h} \|v - v_h\|_H \to 0 \quad \text{as} \quad h \to 0 \quad \text{for all } v \in V.
\end{align*}
Then, we construct an approximation $V_h(k) \subset V_1^0(k)$ by assembling the tensor products of $V_h$, i.e., $V_h(k) = \bigotimes^k V_h$. This suggests $U_h(N) = \bigoplus^N_{k = 1} V_h(k)$. It can be shown that this approximation space fulfills \eqref{assumption:GalerkinApproximationProperty}. We refer to this discretization as the standard discretization since it coincides with the system one would obtain from a discretize-then-linearize approach.

As opposed to existing analyses of the discretize-then-linearize approach, our convergence radii and rates depend on $y_0$, $f$, and $B$ but are independent of $h$. In some instances, similar results can be achieved using the discretize-then-linear approach; see, e.g., \cite{Liu2023EfficientEstimation} for diffusion equations with nonlinear reaction terms. However, for problems in which $B$ includes derivatives, like the problems in Sections~\ref{sec:FlowProblems} and \ref{sec:InteractionTerm}, such an approach fails. This is due to the diverging scaling behaviors of the norms of the discretized operators $A_h$ and $B_h$ with respect to $h$. In particular, for fluid flow problems, the discrete coercivity constant $\gamma_h$ of $A_h$ remains $\mathcal{O}(1)$ but $\|B_h\| = \mathcal{O}(h^{-1})$ as $h \to 0$. This discrepancy ultimately results in a violation of the condition $c_{\mathcal{P}} < \gamma$ in Corollary~\ref{corollary:ConergenceAlphaZero} if $h$ is chosen small enough. For the same reason, the error estimates of \cite{Forets2017ExplicitLinearization,Amini2025CarlemanApproximations} are compromised as $h \to 0$ and do not provide a theoretical guarantee for convergence for small $h$. This was observed in, e.g., \cite{Gonzalez-Conde2025QuantumDynamics}. By isolating the truncation error from the discretization error in our analysis, we bridge this knowledge gap and enhance the understanding of the convergence behavior of the linearization for general infinite-dimensional systems.

\subsection{Non-Standard Discretizations}
A significant drawback of the standard Carleman linearization is that it suffers from the curse of dimensionality. Using the approximation space $V_h(k)$ from the previous section leads to an exponential scaling of the degrees of freedom of $U_h(N)$ in $N$.
This scaling is already problematic for finite-dimensional Cauchy problems of low dimension in combination with moderate regimes of $N$, but depending on the specific dynamical system, this effect can become more severe in the case of infinite-dimensional equations. For PDEs, for example, relatively small models can lead to dimensions in the order of $\operatorname{dim}V_h = \mathcal{O}(10^5)$ to reach satisfactory accuracy, and therefore $\operatorname{dim} U_h(N) = \mathcal{O}(10^{5N})$. Consequently, the Carleman linearization might not be numerically tractable even for $N = 2$.

Instead of using the tensor product spaces, we have more freedom in choosing different approximation spaces for each $k$. High-dimensional equations that involve tensor product and Kronecker sum structures, similar to those found in the Carleman linearization, have been extensively studied. This research provides a wide range of efficient methods specifically designed to address the curse of dimensionality. We refer to the application of such methods as non-standard discretizations.
% These methods often take advantage of the low-rank approximability of the solution.
To demonstrate the flexibility and potential of the linearize-then-discretize approach for non-standard discretizations, we focus on one specific method known as \textit{sparse grids} or \textit{sparse tensor product spaces}. For more information on the general theory and approximation properties of sparse grids; see \cite{Bungartz2004SparseGrids,Garcke2013SparseNutshell,Hochmuth2000TensorApplications,Griebel2012OnSpaces}. For their application to PDEs and Galerkin methods, we refer to \cite{Bungartz1998SparseEquations,Hoang2005High-dimensionalScales}. In addition to showcasing the capability of non-standard discretizations, the use of sparse grids will enable numerical experiments on otherwise intractable problems.

Similarly to the previous discretization, we start with a discretization of $V$ and, additionally, assume that it exhibits a nested sequence of spaces
\begin{align*}
    V_{h_1} \subset V_{h_2} \subset \cdots \subset V_{h_j} \subset \cdots \subset V
\end{align*}
for the decreasing sequence $h_j = 2^{-j}$ for $j \in \mathbb{N}$. Such a hierarchy can, for example, arise from successive mesh refinement when working with PDEs. For the sake of readability, we write $V_j := V_{h_j}$ for indices $j$. It is assumed that the approximation error decreases at the rate
\begin{align*}
    \inf_{v_h \in V_{j}} \| v - v_h\|_H \le c h_{j} \| v \|_{V}
\end{align*}
for some $c > 0$, and the dimensions of the spaces scale as $\dim V_j \sim h_j^{-d}$ for some $d \in \mathbb{N}$.
If we now want to roll out $V_{J}$ for some fixed $J \in \mathbb{N}$ to a discretization of $V_1^0(k)$ and follow the same procedure as in the standard discretization, we would arrive at an exponentially growing space $V_{J}(k)$. Sparse grids, however, suggest that if a function $v \in V_1^0(k)$ exhibits increased regularity of the kind $v \in V^1_0(k)$, a large portion of the elements of $V_{J}(k)$ can be neglected without a significant loss in accuracy. In particular, within such a method one chooses the approximation space
\begin{align*}
    \widehat{V}_{J}(k) := \operatorname{span}\left\{\bigcup_{\vec{j} \in \mathbb{N}^k, |\vec{j}|_1\le J} \bigotimes_{i = 1}^k V_{\vec{j}_i}\right\}.
\end{align*}%
This means that instead of coupling the finest discretization $V_J$ in each dimension within the tensor product, we only couple fine discretizations with coarser ones. This translates to the condition $|\vec{j}|_1 \le J$. The dimensionality of $\widehat{V}_J(k)$ can be bounded by $\operatorname{dim} \widehat{V}_J(k) \lesssim h_J^d J^{k - 1}$ as opposed to $\operatorname{dim} V_J(k) \sim h_J^{kd}$. Finally, we choose $\widehat{U}_h(N) := \bigoplus_{k = 1}^N \widehat{V}_J(k)$.

The additional regularity assumptions can be justified through the following consideration. If we assume that the assumptions of Corollary~\ref{corollary:ConergenceAlphaZero} are fulfilled and also that those of Theorem~\ref{thm:Convergence} hold for $\alpha = 1$, then the unique solution has regularirty $y_N \in W(0, T; V_1^1(N), V_{-1}^1(N)) \subseteq C(0, T; V_0^1(N))$. Increased regularity holds for larger values of $\alpha$, which can be complemented with higher-order sparse grid methods as well.

Another example of structure-exploiting methods is the tensor train decomposition \cite{Oseledets2011Tensor-trainDecomposition}. This method has been successfully applied to problems of Kronecker product form to achieve low-rank approximations. For example, if the underlying nonlinear Cauchy problem is a parabolic PDE, the Carleman linearization describes a system of high-dimensional PDEs. This becomes apparent if we take the example of $H = L^2(\Omega)$ and $V = H^1_0(\Omega)$. Then, it holds that $H(k) = L^2(\Omega^k)$ and $V_1^0(k) = H^1_0(\Omega^k)$, i.e., the tensorized spaces are equivalent to Sobolev spaces in higher-dimensional domains. A survey on the efficacy of tensor train decompositions for high-dimensional PDEs can be found in \cite{Khoromskij2015TensorApplications}. A detailed study of this and other non-standard discretizations would exceed the scope of this paper and is a subject of future research.

% ------
% The order of convergence is a different topic: In certain cases, such as parabolic PDEs, a certain order of convergence can be shown (reference https://epubs.siam.org/doi/10.1137/S0036142900377991). But in most cases it is necessary to assume certain regularity for this. For example, use approximation property with projection and Riesz approximation -> we get nice coefficients for that
% Only comment on Cea's lemma: properties of Riesz operator are nice

\section{Numerical Experiments}
\label{sec:NumericalExperiments}
% Overview of what we want to show
We now verify our theoretical findings with a series of numerical experiments. These experiments include a qualitative analysis of the recovery of nonlinear behavior via the linearization, as well as measuring the convergence properties reflecting the theoretical bounds. Additionally, we present the advantages of the non-standard discretization method employed. Although our primary focus is on a second-order parabolic PDE, it is important to note that the framework we propose is versatile and applicable to a broader range of equations.

The methods are implemented in Python. For the discretization of PDEs, the high-level finite element method library \texttt{DOLFINx} \cite{Baratta2023DOLFINx:Environment} is used. The storage and computation of dense and sparse tensors are done with the tensor compiler suite \texttt{taco}~\cite{Kjolstad2017TheCompiler}. The library \texttt{petsc4py} \cite{Dalcin2011ParallelPython} is used for various numerical linear algebra operations.

\subsection{Burgers' Equation and Discretization}
As a model problem, we consider the one-dimen\-sional Burgers' equation extended by a destabilizing linear term. The PDE in its strong form is given by
\begin{equation*}
\begin{alignedat}{3}
    y'(t, x) - \nu\Delta y(t, x) - \lambda y(t, x) + y(t, x) \frac{\partial}{\partial x} y(t, x) &= f(t, x) \quad &&\text{ for } t \in [0, T) \text{ and } x \in [-1, 1], \\
    y(t, -1) = y(t, 1) &= 0 \quad &&\text{ for } t \in [0, T), \\
    y(0, x) &= y_0(x), \span\span
\end{alignedat}
\end{equation*}
with the viscosity $\nu > 0$ and the coefficient $\lambda \ge 0$ of the destabilizing term. In our notation, the linear operator corresponds to the weak form of $A u = -\nu \Delta u - \lambda u$ and the quadratic operator to $B(u \otimes v) = 1/2 (u v_x + v u_x)$. We use the initial condition
\begin{align}
\label{eq:BurgersInitialCondition}
    y_0(x) = \frac{2 \pi b \sin(\pi x)}{a + \cos(\pi x)}.
\end{align}
for some constants $a > 1$ and $b > 0$. This initial condition is adopted from \cite{Wood2006AnEquation}, wherein it is shown that the Burgers equation with initial condition \eqref{eq:BurgersInitialCondition} admits an exact solution if $\lambda = 0$, $b = \nu$, and $f = 0$:
\begin{align*}
%\label{eq:BurgersExactNoForcing}
    y(t, x) = \frac{2 \pi \nu \exp(-\pi^2 \nu t) \sin(\pi x)}{a + \exp(-\pi^2 \nu t) \cos(\pi x)}.
\end{align*}
We used the exact solution to verify our implementation. In experiments including forcing, we take
\begin{align*}
    f(t, x) = c(x^2 - 1)
\end{align*}
for some constant $c \in \mathbb{R}$.

The Carleman linearization is discretized using continuous piecewise linear finite elements in space and the implicit Euler method in time. It is noted that the time discretization is set to a high accuracy in order to avoid possible interference of the linearization, finite element discretization, and time integration errors. To measure the errors of approximations when there is no known exact solution, we compute a baseline solution using a pseudospectral discretization of the equations. To accommodate high regimes of $N$, we approximate the higher-order terms using the sparse grids technique known as the combination technique~\cite{Griebel2014OnTechnique}. In the case of no forcing, i.e., $c = 0$, the linear system arising from a single implicit Euler step has a block-triangular structure. This allows us to solve the system recursively, where each step involves solving an elliptic PDE linked to varying truncation levels, starting with the block associated with $A_N$, i.e., the right bottom block of $\mathcal{A}_N(t)$. In the case of $c \neq 0$, we have to turn to more elaborate methods since the linear system has a block-triangular structure. Theorem~\ref{thm:Convergence} and Corollary~\ref{corollary:ConergenceAlphaZero} imply a form of block-diagonal dominance of the block operator matrix $\mathcal{A}_N(t)$, which transfers to the linear system of a single implicit Euler step. This property led to the convergence of a block Gauss--Seidel method applied to the linear system. While this property holds in the undiscretized setting, it still has to be shown that it remains true upon discretization. This task is deferred to future research; however, it has been found to hold true empirically, and convergence of the block Gauss--Seidel method is achieved in our numerical experiments.

\subsection{Snapshots of Solutions}

\begin{figure}
    \centering
    \begin{subfigure}{0.49\textwidth}
        \includegraphics[width=\textwidth]{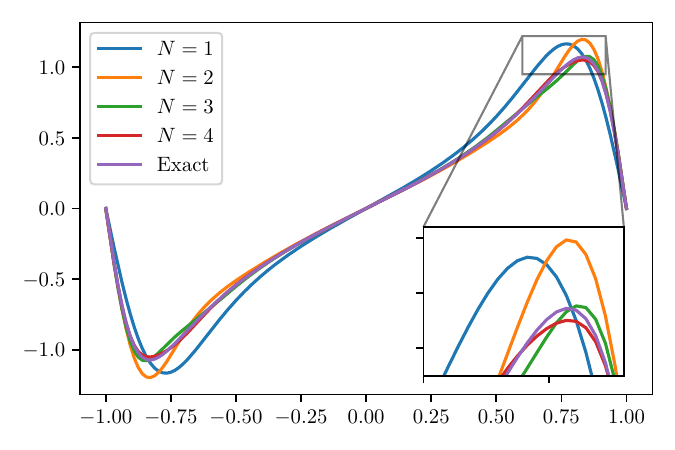}
        \caption{$t = 0.1$}
    \end{subfigure}
    \begin{subfigure}{0.49\textwidth}
        \includegraphics[width=\textwidth]{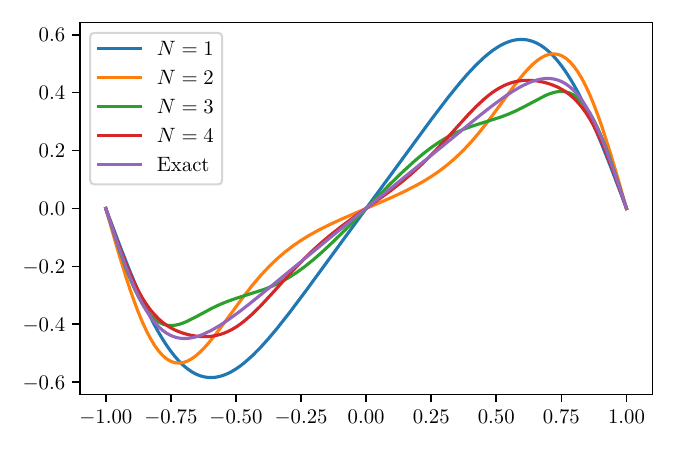}
        \caption{$t = 0.5$}
    \end{subfigure}
    \begin{subfigure}{0.49\textwidth}
        \includegraphics[width=\textwidth]{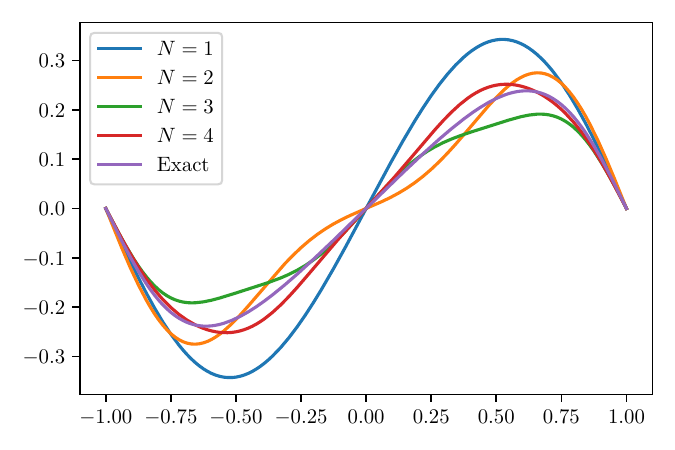}
        \caption{$t = 1$}
    \end{subfigure}
    \begin{subfigure}{0.49\textwidth}
        \includegraphics[width=\textwidth]{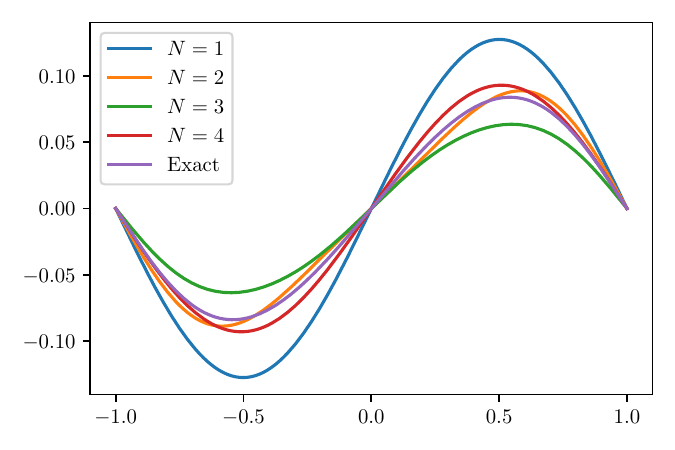}
        \caption{$t = 2$}
    \end{subfigure}
    \caption{Snapshots of solutions to the linearization of the Burgers' equation $y_N^{(1)}(t)$ for different truncation levels $N$ compared to the exact solution $y_e(t)$.}
    \label{fig:Snapshots}
\end{figure}

\begin{figure}
    \centering
    \begin{subfigure}{0.49\textwidth}
        \includegraphics[width=\textwidth]{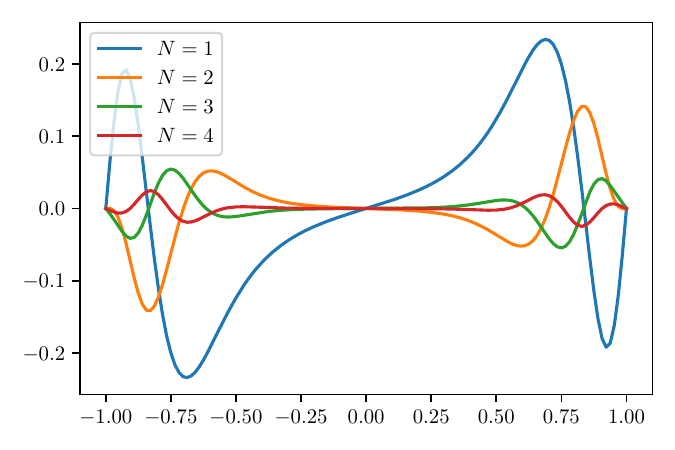}
        \caption{$t = 0.1$}
    \end{subfigure}
    \begin{subfigure}{0.49\textwidth}
        \includegraphics[width=\textwidth]{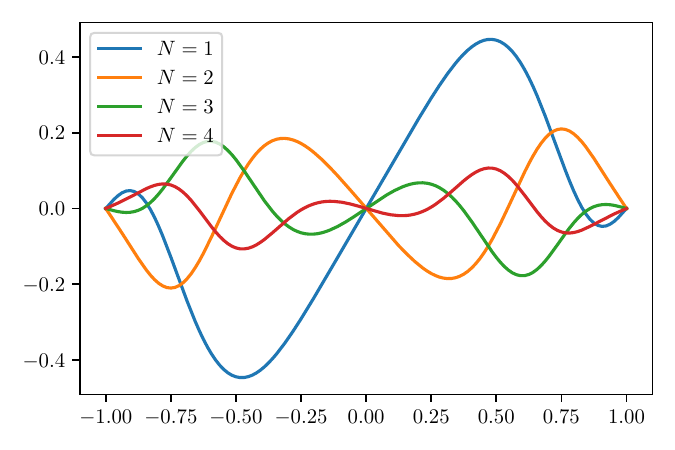}
        \caption{$t = 0.5$}
    \end{subfigure}
    \begin{subfigure}{0.49\textwidth}
        \includegraphics[width=\textwidth]{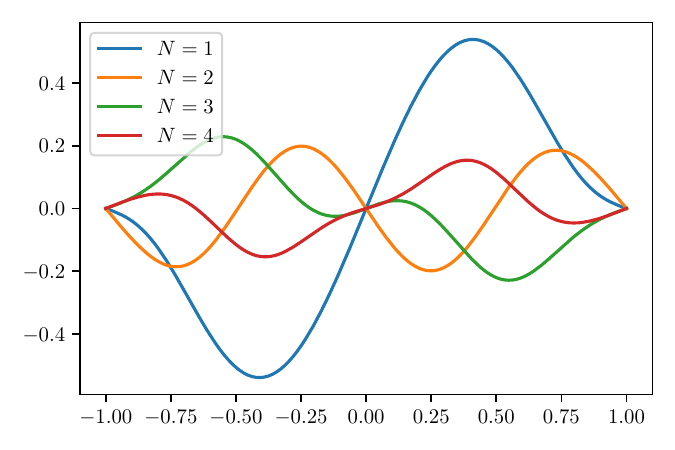}
        \caption{$t = 1$}
    \end{subfigure}
    \begin{subfigure}{0.49\textwidth}
        \includegraphics[width=\textwidth]{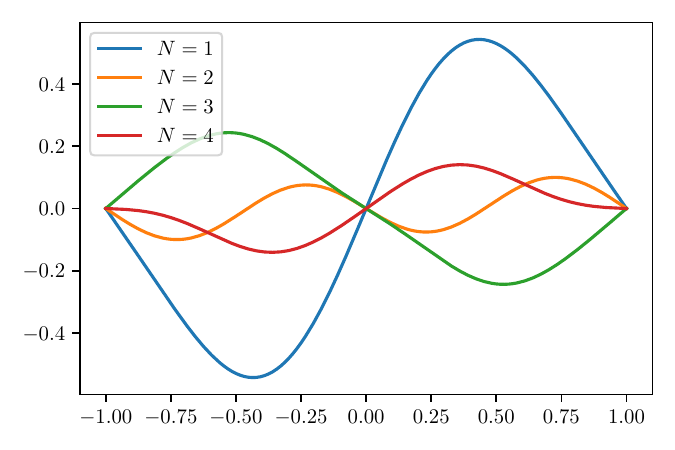}
        \caption{$t = 2$}
    \end{subfigure}
    \caption{Error plots of snapshots of solutions to the linearization of the Burgers' equation $(y_N^{(1)}(t) - y_e(t)) / \|y_e(t)\|_H$ for different truncation levels $N$.}
    \label{fig:ErrorSnapshots}
\end{figure}

First, we analyze snapshots of the exact solution to the Burgers equation in comparison to those obtained through the Carleman linearization. We set $\nu = 0.1$, $a = 1.05$, $b = 0.1$, and $c = 0$. Figure~\ref{fig:Snapshots} shows the solutions at four timestamps for truncation levels up to $N = 4$. At $t = 0.1$, we observe that a higher-order linearization results in a better approximation of the dynamical system. However, it also shows that the quality of the approximation deteriorates over time. This is confirmed in Figure~\ref{fig:ErrorSnapshots}, which illustrates the normalized error given by $(y_N^{(1)}(t) - y_e(t)) / \|y_e(t)\|_H$.

\subsection{Approximation Error}

\begin{figure}
    \centering
    \begin{subfigure}{0.49\textwidth}
    \includegraphics[width=\linewidth]{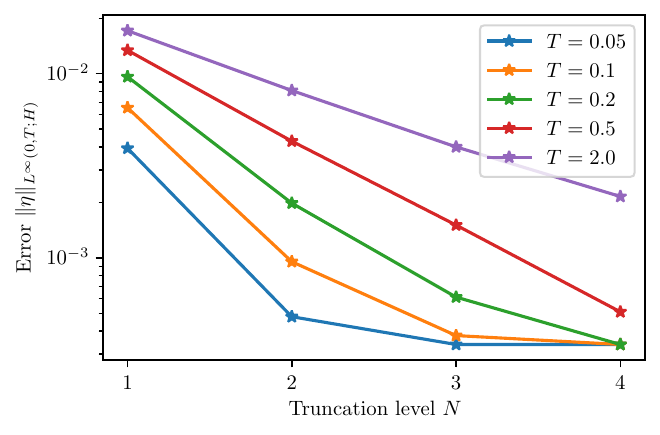}
    \caption{Dependence on $T$}
    \label{fig:ConvergenceParameterT}
    \end{subfigure}
    \begin{subfigure}{0.49\textwidth}
    \includegraphics[width=\linewidth]{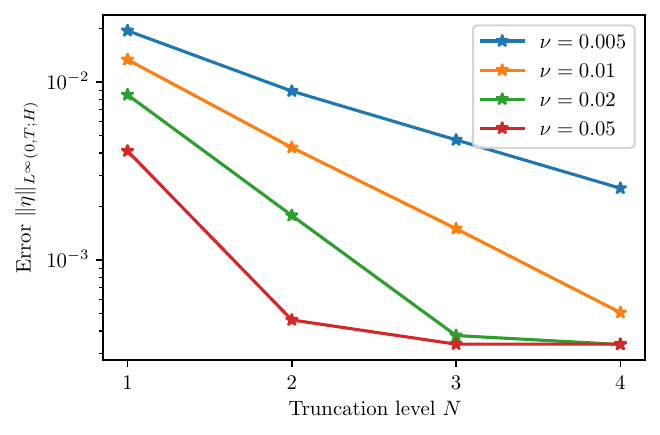}
    \caption{Dependence on $c_B$}
    \label{fig:ConvergenceParameterB}
    \end{subfigure}
    \begin{subfigure}{0.49\textwidth}
    \includegraphics[width=\linewidth]{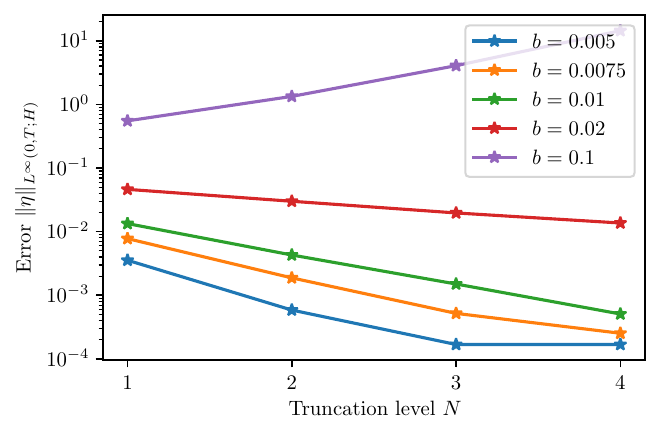}
    \caption{Dependence on $y_0$}
    \label{fig:ConvergenceParameterY0}
    \end{subfigure}
    \begin{subfigure}{0.49\textwidth}
    \includegraphics[width=\linewidth]{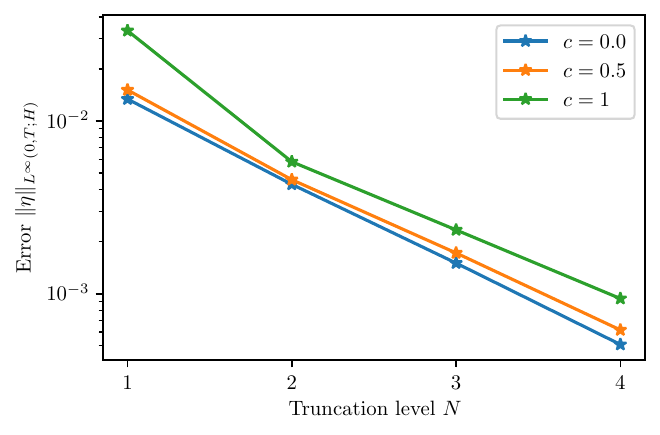}
    \caption{Dependence on $f$}
    \label{fig:ConvergenceParameterF}
    \end{subfigure}
    \begin{subfigure}{0.49\textwidth}
    \includegraphics[width=\linewidth]{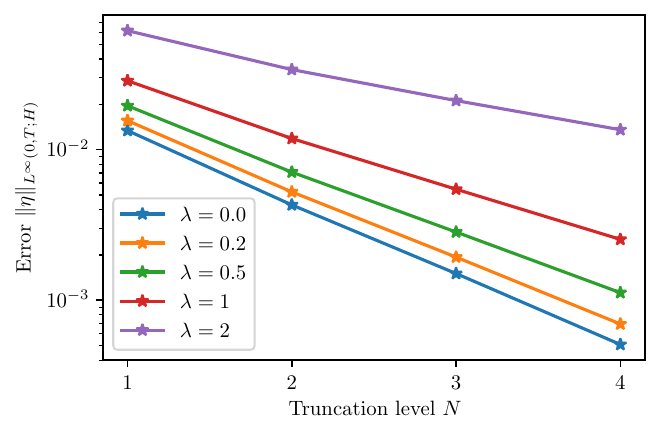}
    \caption{Dependence on $\lambda$ with $T=0.5$}
    \label{fig:ConvergenceParameterLambda}
    \end{subfigure}
    \begin{subfigure}{0.49\textwidth}
    \includegraphics[width=\linewidth]{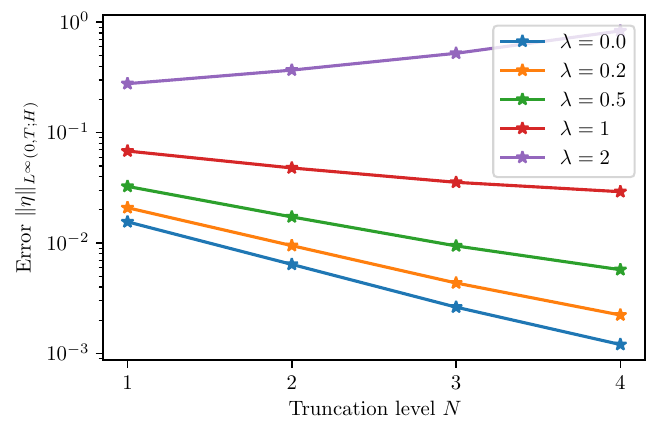}
    \caption{Dependence on $\lambda$ with $T=1$}
    \label{fig:ConvergenceParameterLambdaLongT}
    \end{subfigure}
\caption{Convergence of the Carleman linearization with respect to the truncation level $N$ measured by the error $\|\eta\|_{L^\infty(0, T; H)}$. Each plot indicates that the error behaves exponentially in $N$, whereas different sets of model parameters affect the error bounds in Theorem~\ref{thm:Convergence} and Corollary~\ref{corollary:ConergenceAlphaZero}.}
\end{figure}

Theorem~\ref{thm:Convergence} and Corollary~\ref{corollary:ConergenceAlphaZero} give an upper bound on the expected approximation error arising from the linearization, which suggests (sub-)exponential convergence with respect to $N$. In this section, we verify these convergence rates by analyzing how the error behaves for varying parameters $T$, $c_B$, $y_0$, $f$, and $\lambda$. Each of these parameters are examined in a separate subsection. 

\subsubsection{Dependence on $T$}
The convergence rate given by Theorem~\ref{thm:Convergence} is partially governed  by the $W(0, T; V^{1 + \alpha}, V^{-1 + \alpha})$-norm of $y_e$. Since this expression grows as $T$ is increased, this suggests that slower convergence is to be expected for larger time horizons. Figure~\ref{fig:ConvergenceParameterT} shows the approximation error as a function of $N$ measured with $\|\eta\|_{L^\infty(0, T; H)}$ for $\nu = 0.01$, $a = 1.05$, $b = 0.01$, $c = 0$, $\lambda = 0$, and various final times $T$. It is observed that larger values of $T$ lead to slower convergence. Furthermore, the error decreases exponentially fast with respect to $N$ initially until a certain threshold is reached. This phenomenon reflects the two error source terms, the linearization and the discretization error. The linearization error decays exponentially until the discretization error dominates.

\subsubsection{Dependence on $c_B$}
As observed in the preceding subsection, the convergence rate is determined by the size of the exact solution, which is also affected by the size of the nonlinearity $c_B$. In the case of the Burgers' equation, we have that $c_B \sim \nu$. Figure~\ref{fig:ConvergenceParameterB} depicts the error as a function of $N$ for $T=0.5$, $a = 1.05$, $b = 0.01$, $c = 0$, $\lambda = 0$, and various values for $\nu$. This verifies the deterioration of the convergence rate as $\nu$ is decreased, as well as exponential convergence until the discretization error is reached.

\subsubsection{Dependence on $y_0$}
In addition to the norm of $y_e$, the convergence rate also depends on the initial value $y_0$. Similarly, we expect the convergence to worsen as $y_0$ becomes larger. In our model problem, the size of the initial condition is determined by the parameter $b$. Figure~\ref{fig:ConvergenceParameterY0} shows the error as a function of $N$ for $T=0.5$, $\nu = 0.01$, $a = 1.05$, $c = 0$, $\lambda = 0$, and various values of $b$. The plot confirms the expected behavior and also demonstrates that the linearization diverges when the initial condition is too large.

\subsubsection{Dependence on $f$}
As stated in Corollary~\ref{corollary:ConergenceAlphaZero}, the forcing $f$ has an immediate effect on the size of the solution $y_e$. We expect that larger $f$ result in larger solutions $y_e$, which in turn leads to slower convergence of the linearization. Figure~\ref{fig:ConvergenceParameterF} shows the error as a function of $N$ for $T=0.5$, $\nu = 0.01$, $a = 1.05$, $b = 0.01$, $\lambda = 0$, and various values of $c$, which determines the size of $f$. We observe the expected behavior. It is noted that the size of $f$ has only a mild effect on the convergence in the presented experiments.

\subsubsection{Dependence on $\lambda$}
Lastly, we examine the parameter $\lambda$. The established theoretical error bound includes a factor $\exp(\lambda T)$. This indicates that larger values of $\lambda$ and $T$ lead to poorer convergence. Moreover, the linearization is not guaranteed to converge for arbitrarily large $T$ if $\lambda > 0$. Figures~\ref{fig:ConvergenceParameterLambda} and \ref{fig:ConvergenceParameterLambdaLongT} show the error as a function of $N$ for $\nu = 0.01$, $a = 1.05$, $b = 0.01$, $c = 0$, and various values of $\lambda$. The first plot shows the error for $T = 0.5$, and the second for $T = 1$. In both scenarios, exponential convergence is achieved, which worsens as $\lambda$ is increased. While the linearization converges for all scenarios for $T = 0.5$, the linearization diverges in the case of $T = 1$ and $\lambda = 2$, confirming our error bounds.

% Show that the convergence radius does not change (?)

\subsection{Benefits of Non-Standard Discretization}

\begin{figure}
\centering
\begin{subfigure}{0.49\textwidth}
\includegraphics[width=\textwidth]{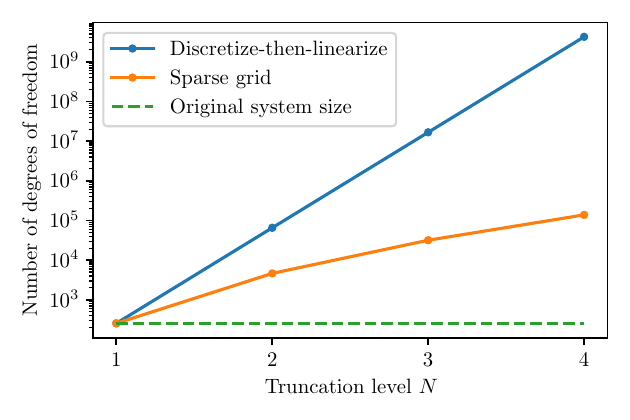}
\caption{DOFs as a function of $N$ for $h = 2^{-8}$}
\label{fig:DegreesOfFreedomN}
\end{subfigure}
\begin{subfigure}{0.49\textwidth}
\includegraphics[width=\textwidth]{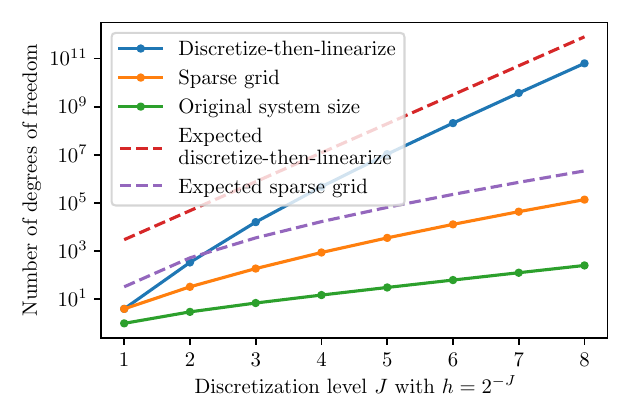}
\caption{DOFs as a function of discret.\ levels for $N = 4$}
\label{fig:DegreesOfFreedomJ}
\end{subfigure}
\caption{Number of degrees of freedom (denoted DOFs) $\dim U_h(N)$ for different discretization methods for varying truncation levels $N$ and refinement levels $J$ (the corresponding maximum cell size of the mesh is given by $h=2^{-J}$).}
\end{figure}

Non-standard discretizations can offer advantages over traditional methods. In our model problem, the use of sparse grids improves the scaling of the required degrees of freedom. Figure~\ref{fig:DegreesOfFreedomN} illustrates the spatial degrees of freedom $\dim U_h(N)$ for various values of $N$ using the selected discretization. Meanwhile, Figure~\ref{fig:DegreesOfFreedomJ} shows the scaling (in some cases expected scaling) of $\dim U_h(N)$ for different levels of discretization $J$, where the maximum cell size of the mesh is determined by $h = 2^{-J}$. These plots demonstrate the benefits of the non-standard discretization compared to the discretize-then-linearize approach, highlighting how this method helps alleviate the exponential increase in degrees of freedom.

\section{Conclusion}
\label{sec:Conclusion}
% Add reference to parallel-in-time paper as future work

In this paper, we derived the well-posedness and the convergence of the truncated Carleman linearization for infinite-dimensional parabolic Cauchy problems under suitable assumptions. We achieved this in an undiscretized setting. This allowed us to separate the error arising from the truncated linearization from the discretization error stemming from the approximation of the infinite-dimensional equations. We thus justified the application of the linearization to PDE problems. We verified the theoretical findings with a series of numerical experiments, showing the expected convergence of the linearization. The theoretical findings motivate and support the application of non-standard discretization methods, which enable higher-order linearizations that were previously intractable. Numerical experiments showcase how such methods can reduce the number of degrees of freedom by orders of magnitude.

This paper addresses the linearization of parabolic dynamical systems. Even though parabolic equations possess several favorable regularity properties and are generally better understood than hyperbolic equations, it is suggested that many of the concepts discussed here are transferable to second-order hyperbolic systems of the form
\begin{align*}
y''(t) + A y(t) + B (y(t) \otimes y(t)) &= f(t) \quad\text{in } L^2(0, T; V'), \\
y(0) &= y_0, \\
y'(0) &= y_1.
\end{align*}
By assuming analogous properties for $A$, $B$, and the Hilbert spaces, one can derive robust coercivity and boundedness constants for the linear system associated with the Carleman linearization, following the same argumentation presented in this paper. This approach may yield analogous well-posedness and convergence results for the linearized problem, thereby facilitating a rigorous theoretical analysis of hyperbolic nonlinear PDEs.

While this paper introduces methods to mitigate the computational burden of the Carleman linearization, the practical performance and optimization of these methods remain open questions. A study on the computational aspects of the Carleman linearization could evaluate the effectiveness of sparse grid methods across various applications, explore their efficient implementation, and investigate the possibility of establishing theoretical bounds on the discretization error. Additionally, such a study could examine various structure-exploiting low-rank methods such as the tensor train method. The performance of the methods could be analyzed in various applications of the linearization, including model-order reduction and optimal control design.

Another potential branch of future research is concerned with the efficient solution of parabolic Cauchy problems, as well as associated optimal control problems. So-called parallel-in-time methods address the efficient solution of dynamical systems while enabling parallelization along the time axis. Among these methods, diagonalization-based approaches, such as those discussed in \cite{Gander2021ParaDiag:Technique,Wu2020Diagonalization-basedProblems,Heinzelreiter2024Diagonalization-basedProblems}, have shown significant promise. These methods excel in handling (nearly) linear time-invariant problems. Recent studies have demonstrated the efficiency of diagonalization-based approaches from moderately sized problems to complex linear fluid flow problems. However, they encounter difficulties with nonlinear equations. Integrating the Carleman linearization with non-standard discretizations could potentially extend the applicability of diagonalization-based methods to nonlinear PDE problems, including the Navier--Stokes equations.

% Lastly, a well-established method for solving large-scale nonlinear equations is Newton's method. For highly nonlinear equations, such as the Navier--Stokes equation with a high Reynolds number, Newton's method can struggle to converge unless a good initial guess is provided. Newton's method is based on a local second-order approximation of a nonlinear equation. The Carleman linearization could enhance Newton's method by generalizing it to a higher-order approach, thereby improving its convergence properties.

\section*{Acknowledgements}
This work has made use of the resources provided by the Edinburgh Compute and Data Facility (ECDF) (\url{http://www.ecdf.ed.ac.uk/}). BH was supported by the MAC--MIGS Centre for Doctoral Training under EPSRC grant EP/S023291/1. JWP was supported by the EPSRC grant EP/Z533786/1. The authors would like to thank Stefan Klus for insightful discussions on linearization techniques and Tobias Breiten for his valuable input on bilinear systems.

\section*{References}
\printbibliography[heading=none]

\appendix
\section{Proof of Lemma~\ref{lem:WellPosednessNonlinear}}
\label{appendix:ProofsNonlinearCauchyProblem}

Before we proceed to the proof of Lemma~\ref{lem:WellPosednessNonlinear}, we introduce a few results on $B$. Most of these are adaptations of results proved in \cite{Breiten2019FeedbackApproximation} with a focus on the explicit dependence of the involved constants on $T$. In the following, we adopt the notation $B(x, y)$ in place of $B(x \otimes y)$ to emphasize the bilinear structure of the operator.
\begin{lemma}
    \label{lem:WBound}
    Let $T \in (0, \infty]$. For all $z \in W(0, T; V, V')$, it holds that
    \begin{align*}
        \|z\|_{L^\infty(0, T; H)} \le \|z(0)\|_H + \|z\|_{W(0, T; V, V')}.
    \end{align*}
\end{lemma}
\begin{proof}
    Let $z \in W(0, T; V, V')$. Then, we obtain the following upper bound:
\begin{align*}
    \|z(t)\|_H^2 &= \| z(0) \|_H^2 + 2 \int_0^t \langle z'(s), z(s) \rangle_V \mathrm{d}s \le \| z(0) \|_H^2 + 2 \int_0^T \left|\langle z'(s), z(s) \rangle_V\right| \mathrm{d}s \\
    &\le \| z(0) \|_H^2 + 2 \|z'\|_{L^2(0, T; V')} \|z\|_{L^2(0, T; V)} \le \| z(0) \|_H^2 + \|z'\|_{L^2(0, T; V')}^2 + \|z\|_{L^2(0, T; V)}^2
\end{align*}
for $t \in [0, T)$, cf.~Lemma III.1.2 in \cite{Temam1979Navier-StokesAnalysis}.
\end{proof}

\begin{lemma}[Lemma 2 in \cite{Breiten2019FeedbackApproximation}]
\label{lem:BLInfinityBound}
    Let $T \in (0, \infty]$ and $B$ fulfill \eqref{assumption:BBounded}. Then, for all $y, z, w \in W(0, T; V, V')$, it holds that
    \begin{align*}
        &\left| \langle B(y, z), w\rangle_{L^2(0, T; V'), L^2(0, T; V)} \right| \\
        &\quad\le c_B \|y\|_{L^\infty(0, T; H)}^{\frac{1}{2}} \|y\|_{L^2(0, T; V)}^{\frac{1}{2}} \|z\|_{L^\infty(0, T; H)}^{\frac{1}{2}} \|z\|_{L^2(0, T; V)}^{\frac{1}{2}} \|w\|_{L^2(0, T; V)}.
    \end{align*}
\end{lemma}

\begin{lemma}[Refinement of Corollary 3 in \cite{Breiten2019FeedbackApproximation}]
\label{lem:BoundedBWNorm}
    Let $T \in (0, \infty]$ and $B$ fulfill \eqref{assumption:BBounded}. For all $y, z \in W(0, T; V, V')$, it holds that
    \begin{align*}
        \|B(y, z)\|_{L^2(0, T; V')} \le c_B (\|y(0)\|_H + \|y\|_{W(0, T; V, V')}) (\|z(0)\|_H + \|z\|_{W(0, T; V, V')}).
    \end{align*}
\end{lemma}
\begin{proof}
    The result follows from Lemma~\ref{lem:WBound} and Lemma~\ref{lem:BLInfinityBound}.
\end{proof}

\begin{lemma}[Refinement of Lemma 4 in \cite{Breiten2019FeedbackApproximation}]
\label{lem:BLipschitz}
    Let $T \in (0, \infty]$ and $B$ fulfill \eqref{assumption:BBounded}. For all $\delta \in [0, 1]$ and for all $y, z \in W(0, T; V, V')$ with $\|y(0)\|_H + \|y\|_{W(0, T; V, V')} \le \delta$ and $\|z(0)\|_H + \|z\|_{W(0, T; V, V')} \le \delta$, it holds that
    \begin{align*}
        \| B(y, y) - B(z, z) \|_{L^2(0, T; V')} \le 2 \delta c_B \left(\|y(0) - z(0)\|_H + \| y - z\|_{W(0, T; V, V')}\right).
    \end{align*}
\end{lemma}
\begin{proof}
Let $y, z \in W(0, T; V, V')$. With Lemma~\ref{lem:BoundedBWNorm}, it follows that
\begin{align*}
\|B(y, y) - B(z, z)\|_{L^2(0, T; V')} &\le \|B(y, y - z)\|_{L^2(0, T; V')} + \|B(y - z, z)\|_{L^2(0, T; V')} \\
&\le 2 \delta c_B \left(\|y(0) - z(0)\|_H + \|y - z\|_{W(0, T; V, V')}\right). \qedhere
\end{align*}
\end{proof}

This enables us to prove the well-posedness of the nonlinear Cauchy problem.
\begin{proof}[Proof of Lemma~\ref{lem:WellPosednessNonlinear}]
Let $y_0 \in H$ and $g \in L^2(0, T; V')$. Then, the system
\begin{equation}
\label{eq:LemmaNonlinearWellposednessProofLinearCauchy}
\begin{aligned}
    z'(t) + A z(t) &= g(t) \quad \text{in } L^2(0, T; V'), \\
    z(0) &= y_0
\end{aligned}
\end{equation}
has a unique solution $z \in W(0, T; V, V')$ with $\|z(0)\|_H + \| z \|_{W(0, T; V, V')} \le c_L \exp(\lambda T)(\|y_0\|_H + \|g\|_{L^2(0, T; V')})$ with the constant $c_L \ge 1$ from Lemma~\ref{lem:WellPosednessLinear}. Let $\mu := \|y_0\|_H + \|f \|_{L^2(0, T; V')}$.
We set $c_N = \max(1/(4 c_B \exp(\lambda T)), c_L) \ge 1$. Define the set
\begin{equation*}
M = \{ y \in W(0, T; V, V') \mid \|y(0)\|_H + \|y\|_{W(0, T; V, V')} \le 2 c_N \exp(\lambda T)\mu, ~~ y(0) = y_0 \}.    
\end{equation*}
Due to $c_N \ge c_L$ and estimate \eqref{eq:ParabolicEstimateWNorm}, the solution to \eqref{eq:LemmaNonlinearWellposednessProofLinearCauchy} with $g = f$ belongs to $M$. Thus, $M$ is not empty. Due to the continuous embedding $C(0, T; H) \hookrightarrow W(0, T; V, V')$, the set is closed under the $W(0, T; V, V')$--norm. Define the map $Z: M \to W(0, T; V, V')$, where $z = Z(y)$ maps the function $y$ to the solution of the system
\begin{align*}
    z'(t) + A z(t) + B(y(t), y(t))  &= f(t), \\
    z(0) &= y_0.
\end{align*}
Since $c_N \ge c_L$ and by using Lemma~\ref{lem:BLipschitz} with $\delta = 2 c_N \exp(\lambda T) \mu \le \frac{1}{4 c_N \exp(\lambda T) c_B} \le 1$ and one argument being zero, we obtain
\begin{align*}
    \|z(0)\|_H + \|z\|_{W(0, T; V, V')} &\le c_N \exp(\lambda T) \left( \|y_0\|_H + \|B(y, y)\|_{L^2(0, T; V')} + \|f\|_{L^2(0, T, V')} \right) \\
    &\le c_N \exp(\lambda T) \left(\mu + 2 \delta c_B \left(\|y_0\| + \|y\|_{W(0, T; V, V')}\right)\right) \\
    &\le c_N \exp(\lambda T) \left( \mu + \frac{1}{2 c_N \exp(\lambda T) c_B} 2 c_B c_N \exp(\lambda T) \mu \right) = 2 c_N \exp(\lambda T) \mu.
\end{align*}
Hence, $Z(M) \subseteq M$. Next, we show that $Z$ is a contraction. For $y_1, y_2 \in M$, let $z = Z(y_1) - Z(y_2)$, which solves
\begin{align*}
    z'(t) + A z(t) + B(y_1(t), y_1(t)) - B(y_2(t), y_2(t)) &= 0 \quad \text{in } L^2(0, T; V'), \\
    z(0) &= 0.
\end{align*}
From Lemma~\ref{lem:BLipschitz}, we obtain that
\begin{align*}
    &\|Z(y_1) - Z(y_2)\|_{W(0, T; V, V')} = \| z \|_{W(0, T; V, V')} \le c_N \exp(\lambda T) \left( \|B(y_1, y_1) - B(y_2, y_2) \|_{L^2(0, T; V')} \right) \\
    &\quad\le c_N \exp(\lambda T) 2 \delta c_B \|y_1 - y_2\|_{W(0, T; V, V')} \le \frac{1}{2} \|y_1 - y_2\|_{W(0, T; V, V')}.
\end{align*}
Due to Banach's fixed-point theorem, there is a unique solution $y \in M$ to $Z(y) = y$, which proves the existence of a solution to the nonlinear Cauchy problem. 

The uniqueness of the solution in $W(0, T; V, V')$ can be proven the same way as in \cite{Breiten2019FeedbackApproximation}.
\end{proof}

\section{Proofs of Properties of $A_k$, $B_k$, and $F_k$}
\label{appendix:ProofsAB}

\newcommand{\skipindex}[2]{\vec{#1}^{(#2)}}
\newcommand{\replace}[3]{%
  \vec{#1}^{(#2; #3)}%
}

For the following proofs, we introduce the notation $\skipindex{j}{l} := (\vec{j}_1, \dots, \vec{j}_{l - 1}, \vec{j}_{l + 1}, \dots, \vec{j}_{k}) \in \mathbb{N}^{k - 1}$ for $k \ge 2$ and $\replace{j}{l}{p_1, \dots, p_m} := (\vec{j}_1, \dots, \vec{j}_{l - 1}, p_1, \dots, p_m, \vec{j}_{l + 1}, \dots, \vec{j}_{k}) \in \mathbb{N}^{k + m - 1}$ for $k \ge 1$, where $\vec{j} \in \mathbb{N}^{k}$ is a multi-index with $k \in \mathbb{N}$, and $p_n \in \mathbb{N}$ for $n \in \{1, \dots, m\}$ and $m \in \mathbb{N}$. Furthermore, the notation $\sum_{\skipindex{j}{l} \in \mathbb{N}^{k - 1}}$ for $l \in \{1, \dots, k\}$ translates to the sum over $\mathbb{N}^{k - 1}$ with the indices $\skipindex{j}{l} = (\vec{j}_1, \dots, \vec{j}_{l - 1}, \vec{j}_{l + 1}, \dots, \vec{j}_{k - 1})$. The multi-index $\replace{j}{l}{p_1, \dots, p_m}$ is defined as above in such case. We will use $\langle \cdot, \cdot \rangle$ to denote the duality mappings $\langle \cdot, \cdot\rangle_V$ and $\langle \cdot, \cdot\rangle_{V_1^0(k)}$, where the respective spaces are to be inferred from the context.

\begin{proof}[Proof of Lemma~\ref{lem:Ak}]
First, we show the operator's boundedness. Let $u, v \in V_1^0(k)$ be arbitrary but fixed. 
It holds that
\begin{align*}
    A_k v = \sum_{\vec{i} \in \mathbb{N}^k} \langle A_k v, \varphi_{\vec{i}} \rangle \varphi_{\vec{i}}.
\end{align*}
The functions $u$ and $v$ admit the representation $u = \sum_{\vec{i} \in \mathbb{N}^k} \hat{u}(\vec{i}) \varphi_{\vec{i}}$ with $\hat{u}(\vec{i}) = \langle u, \varphi_{\vec{i}}\rangle$, and a similar expression for $v$. It holds that
\begin{align*}
\langle A_k u, v \rangle &= \sum_{l = 1}^k \sum_{\vec{i} \in \mathbb{N}^k} \sum_{\vec{j} \in \mathbb{N}^k} \hat{u}(\vec{i}) \hat{v}(\vec{j}) \left\langle \left( \bigotimes_{m = 1}^{l - 1} \varphi_{\vec{i}_m}\right) \otimes A\varphi_{\vec{i}_l} \otimes \left( \bigotimes_{m = 1}^{k - l} \varphi_{\vec{i}_m}\right), \varphi_{\vec{j}} \right\rangle \\
&= \sum_{l = 1}^k \sum_{\skipindex{i}{l} \in \mathbb{N}^{k - 1}} \sum_{p, r \in \mathbb{N}} \hat{u}(\replace{i}{l}{p}) \hat{v}(\replace{i}{l}{r}) \langle A \varphi_p, \varphi_r \rangle.
\end{align*}
Define the functions $g_{\vec{i}, l} = \sum_{s \in \mathbb{N}} \hat{u}(\replace{i}{l}{s}) \varphi_s \in V$ and $w_{\vec{i}, l} = \sum_{s \in \mathbb{N}} \hat{v}(\replace{i}{l}{s}) \varphi_s \in V$. Then, by leveraging the Cauchy--Schwarz inequality, we obtain that
{\allowdisplaybreaks
\begin{align*}
    \langle A_k u, v \rangle &= \sum_{l = 1}^k \sum_{\skipindex{i}{l} \in \mathbb{N}^{k - 1}} \langle A g_{\vec{i}, l}, w_{\vec{i}, l} \rangle \underset{\eqref{assumption:ABounded}}{\le} \sum_{l = 1}^k \sum_{\skipindex{i}{l} \in \mathbb{N}^{k - 1}} \beta \| g_{\vec{i}, l} \|_{V} \| w_{\vec{i}, l} \|_{V} \\
    &\le \beta \left[ \sum_{l = 1}^k \sum_{\skipindex{i}{l} \in \mathbb{N}^{k - 1}} \| g_{\vec{i}, l} \|_{V}^2 \right]^{\frac{1}{2}} \left[ \sum_{l = 1}^k \sum_{\skipindex{i}{l} \in \mathbb{N}^{k - 1}} \| w_{\vec{i}, l} \|_{V}^2 \right]^{\frac{1}{2}} \\
    &= \beta \left[ \sum_{l = 1}^k \sum_{\skipindex{i}{l} \in \mathbb{N}^{k - 1}} \sum_{s \in \mathbb{N}} \lambda_s \hat{u}(\replace{i}{l}{s})^2 \right]^{\frac{1}{2}} \left[ \sum_{l = 1}^k \sum_{\skipindex{i}{l} \in \mathbb{N}^{k - 1}} \sum_{s \in \mathbb{N}} \lambda_s \hat{v}(\replace{i}{l}{s})^2 \right]^{\frac{1}{2}} \\
    &= \beta \left[ \sum_{l = 1}^k \sum_{\vec{i} \in \mathbb{N}^{k}} \lambda_{\vec{i}_l} \hat{u}(\vec{i})^2 \right]^{\frac{1}{2}} \left[ \sum_{l = 1}^k \sum_{\vec{i} \in \mathbb{N}^{k}} \lambda_{\vec{i}_l} \hat{v}(\vec{i})^2 \right]^{\frac{1}{2}} = \beta \| u \|_{V_1^0(k)} \| v \|_{V_1^0(k)},
\end{align*}
}%
which shows the boundedness of $A_k$.

Next, we show the coercivity of $A_k$ by
{\allowdisplaybreaks
\begin{align*}
\langle A_k v, v\rangle &= \sum_{l = 1}^k \sum_{\vec{i} \in \mathbb{N}^k} \sum_{\vec{j} \in \mathbb{N}^k} \hat{v}(\vec{i}) \hat{v}(\vec{j}) \left\langle \left( \bigotimes^{l - 1}_{m = 1} \varphi_{\vec{i}_m} \right) \otimes A \varphi_{\vec{i}_l} \otimes \left( \bigotimes^{k - l}_{m = 1} \varphi_{\vec{i}_m} \right), \varphi_{\vec{j}}\right\rangle \\
& = \sum_{l = 1}^k \sum_{\skipindex{i}{l} \in \mathbb{N}^{k - 1}} \sum_{p, r \in \mathbb{N}} \hat{v}(\replace{i}{l}{p}) \hat{v}(\replace{i}{l}{r}) \langle A\varphi_{p}, \varphi_{r}\rangle = \sum_{l = 1}^k \sum_{\skipindex{i}{l} \in \mathbb{N}^{k - 1}} \langle A w_{\vec{i}, l}, w_{\vec{i}, l}\rangle \\
&\underset{\eqref{assumption:AVHCoercive}}{\ge} \sum_{l = 1}^k \sum_{\skipindex{i}{l} \in \mathbb{N}^{k - 1}} \left(\gamma \|w_{\vec{i}, l}\|_V^2 - \lambda \| w_{\vec{i}, l} \|_H^2\right) \\
&= \gamma \sum_{l = 1}^k \sum_{\skipindex{i}{l} \in \mathbb{N}^{k - 1}} \sum_{s \in \mathbb{N}} \lambda_s \hat{v}(\replace{i}{l}{s})^2 - \lambda \sum_{l = 1}^k \sum_{\skipindex{i}{l} \in \mathbb{N}^{k - 1}} \sum_{s \in \mathbb{N}} \hat{v}(\replace{i}{l}{s})^2 \\
&= \gamma \sum_{l = 1}^k \sum_{\vec{i} \in \mathbb{N}^k} \lambda_{\vec{i}_l} \hat{v}(\vec{i})^2 - \lambda \sum_{l = 1}^k \sum_{\vec{i} \in \mathbb{N}^k} \hat{v}(\vec{i})^2 = \gamma \|v \|_{V_0^1(k)}^2 - k \lambda \|v\|_{H(k)}^2. \qedhere
\end{align*}
}%
\end{proof}

\begin{proof}[Proof of Lemma~\ref{lem:Bk}]
Let $B_m^{p, r} = \langle B(\varphi_p \otimes \varphi_r), \varphi_m\rangle$ for $p, r, m \in \mathbb{N}$. Any $v \in V_1^\alpha(k + 1)$ can be represented as $v = \sum_{\vec{i} \in \mathbb{N}^{k + 1}} \hat{v}(\vec{i}) \varphi_{\vec{i}}$. Then,
\begin{align*}
    B_k v &= \sum_{\vec{j} \in \mathbb{N}^k} \langle B_k v, \varphi_{\vec{j}} \rangle \varphi_{\vec{j}}.
\end{align*}
Moreover,
\begin{align*}
    \langle B_k v, \varphi_{\vec{j}} \rangle = \sum_{l = 1}^k \left\langle \left(\bigotimes^{l - 1} I\right) \otimes B \otimes \left(\bigotimes^{k - l} I\right) v, \varphi_{\vec{j}} \right\rangle = \sum_{l = 1}^k \sum_{p, r \in \mathbb{N}} B_{\vec{j}_l}^{p, r} \hat{v}(\replace{j}{l}{p, r}).
\end{align*}
Using the Cauchy--Schwarz and H\"older inequalities, we can conclude that
\begin{align*}
    &\| B_k v \|_{V_{-1 + \varepsilon}^\alpha(k)}^2 = \sum_{\vec{j} \in \mathbb{N}^k} \pi(\lambda_{\vec{j}})^\alpha \sigma(\lambda_{\vec{j}})^{-1 + \varepsilon} \left| \sum_{l = 1}^k \sum_{p, r \in \mathbb{N}} B_{\vec{j}_l}^{p, r} \hat{v}(\replace{j}{l}{p, r}) \right|^2 \\
    &\quad\le \sum_{\vec{j} \in \mathbb{N}^k} \pi(\lambda_{\vec{j}})^\alpha \underbrace{\sigma(\lambda_{\vec{j}})^{-1 + \varepsilon} \left( \sum_{l = 1}^k \lambda_{\vec{j}_l}^{1 - \varepsilon}  \right)}_{ \le k^\varepsilon } \sum_{l = 1}^k \frac{1}{\lambda_{\vec{j}_l}^{1 - \varepsilon}} \left| \sum_{p, r \in \mathbb{N}} B_{\vec{j}_l}^{p, r} \hat{v}(\replace{j}{l}{p, r}) \right|^2 \\
    &\quad\le k^\varepsilon \sum_{\vec{j} \in \mathbb{N}^k} \pi(\lambda_{\vec{j}})^\alpha \sum_{l = 1}^k \frac{1}{\lambda_{\vec{j}_l}^{1 - \varepsilon}} \left| \sum_{p, r \in \mathbb{N}} B_{\vec{j}_l}^{p, r} \hat{v}(\replace{j}{l}{p, r}) \right|^2.
\end{align*}
H\"older's inequality was used with $p = 1/(1 - \varepsilon)$ and $q = 1 / \varepsilon$ for $\varepsilon \in [0, 1)$. Then, $1/p + 1/q = 1$ and
\begin{align*}
    \sum_{l = 1}^k \lambda_{\vec{j}_l}^{1 - \varepsilon} \le \left( \sum_{l = 1}^k \lambda_{\vec{j}_l}^{(1 - \varepsilon) p} \right)^{\frac{1}{p}} \left(\sum_{l = 1}^k 1 \right)^\frac{1}{q} = \left(\sum_{l = 1}^k \lambda_{\vec{j}_l} \right)^{1 - \varepsilon} k^\varepsilon.
\end{align*}
The same estimate can be shown directly for $\varepsilon = 1$.
For $\vec{j} \in \mathbb{N}^k$ fixed, let $w_{\vec{j}, l} := \allowbreak \sum_{p, r \in \mathbb{N}}  \allowbreak \hat{v}(\replace{j}{l}{p, r}) \varphi_p \otimes \varphi_r \in V^\alpha_1(2)$. From the regularity of the operator $B$, we have that
\begin{align*}
    &\sum_{\vec{j}_l \in \mathbb{N}} \lambda_{\vec{j}_l}^{\alpha - 1 + \varepsilon} \left| \sum_{p, r \in \mathbb{N}} B_{\vec{j}_l}^{p, r} \hat{v}(\replace{j}{l}{p, r}) \right|^2 = \| B w_{\vec{j}, l} \|_{V^{\alpha - 1 + \varepsilon}}^2 \\
    &\quad\underset{\eqref{assumption:BBilinearMapping}}{\le} c_B(\alpha, \varepsilon)\|w_{\vec{j}, l}\|_{V_1^\alpha(2)}^2
    = c_B(\alpha, \varepsilon) \sum_{p, r \in \mathbb{N}} \hat{v}(\replace{j}{l}{p, r})^2 \lambda_p^\alpha \lambda_r^\alpha (\lambda_p + \lambda_r).
    % &= c \sum_{\vec{j} \in \mathbb{N}^{k + 1}} \left| \hat{v}(\vec{j}) \right|^2 \lambda_{j_l}^\alpha \lambda_{j_{l + 1}}^\alpha (\lambda_{j_l} + \lambda_{j_{l + 1}}).
\end{align*}
Plugging this into the above estimate, we obtain
{\allowdisplaybreaks
\begin{align*}
    &\sum_{\vec{j} \in \mathbb{N}^k} \pi(\lambda_{\vec{j}})^\alpha \sum_{l = 1}^k \frac{1}{\lambda_{\vec{j}_l}^{1 - \varepsilon}} \left| \sum_{p, r \in \mathbb{N}} B_{\vec{j}_l}^{p, r} \hat{v}(\replace{j}{l}{p, r}) \right|^2 \\
    &\quad= \sum_{l = 1}^k \sum_{\skipindex{j}{l} \in \mathbb{N}^{k - 1}} \sum_{\vec{j}_l \in \mathbb{N}} \pi(\lambda_{\vec{j}})^\alpha \frac{1}{\lambda_{\vec{j}_l}^\alpha} \lambda_{\vec{j}_l}^{\alpha - 1 + \varepsilon} \left| \sum_{p, r \in \mathbb{N}} B_{\vec{j}_l}^{p, r} \hat{v}(\replace{j}{l}{p, r}) \right|^2 \\
    &\quad= \sum_{l = 1}^k \sum_{\skipindex{j}{l} \in \mathbb{N}^{k - 1}} \left( \prod_{m \neq l} \lambda_{\vec{j}_m}^\alpha \right) \sum_{\vec{j}_l \in \mathbb{N}} \lambda_{\vec{j}_l}^{\alpha - 1 + \varepsilon} \left| \sum_{p, r \in \mathbb{N}} B_{\vec{j}_l}^{p, r} \hat{v}(\replace{j}{l}{p, r}) \right|^2 \\
    &\quad\le c_B(\alpha, \varepsilon) \sum_{l = 1}^k \sum_{\skipindex{j}{l} \in \mathbb{N}^{k - 1}} \left( \prod_{m \neq l} \lambda_{\vec{j}_m}^\alpha \right) \sum_{p, r \in \mathbb{N}} \hat{v}(\replace{j}{l}{p, r})^2 \lambda_p^\alpha \lambda_r^\alpha (\lambda_p + \lambda_r) \\
    &\quad= c_B(\alpha, \varepsilon) \sum_{l = 1}^k \sum_{\vec{i} \in \mathbb{N}^{k + 1}} \pi(\lambda_{\vec{i}})^\alpha \hat{v}(\vec{i})^2 (\lambda_{\vec{i}_l} + \lambda_{\vec{i}_{l + 1}})
    = c_B(\alpha, \varepsilon) \sum_{\vec{i} \in \mathbb{N}^{k + 1}} \pi(\lambda_{\vec{i}})^\alpha \hat{v}(\vec{i})^2 \underbrace{\sum_{l = 1}^k (\lambda_{\vec{i}_l} + \lambda_{\vec{i}_{l + 1}})}_{\le 2\sigma(\lambda_{\vec{i}})} \\
    &\quad\le 2 c_B(\alpha, \varepsilon) \sum_{\vec{i} \in \mathbb{N}^{k + 1}} \pi(\lambda_{\vec{i}})^\alpha \sigma(\lambda_{\vec{i}}) \hat{v}(\vec{i})^2 = 2 c_B(\alpha, \varepsilon) \| v \|_{V_1^\alpha(k + 1)}^2.
\end{align*}
Thus, 
\begin{equation*}
    \| B_k v \|_{V_{-1 + \varepsilon}^\alpha(k)}^2 \le 2 c_B(\alpha, \varepsilon) k^\varepsilon \| v \|_{V_1^\alpha(k + 1)}^2. \qedhere
\end{equation*}
}%

\end{proof}

\begin{proof}[Proof of Lemma~\ref{lem:Fk}]
For all $v \in V_1^0(k - 1)$ it holds that
\begin{align*}
        F_k(t) v = \sum_{\vec{j} \in \mathbb{N}^{k}} \langle F_k(t) v, \varphi_{\vec{j}}\rangle \varphi_{\vec{j}} = \sum_{\vec{j} \in \mathbb{N}^{k}} \varphi_{\vec{j}} \sum_{l = 1}^k \langle -f, \varphi_{\vec{j}_l}\rangle \hat{v}(\skipindex{j}{l}).
\end{align*}
With $\lambda_{\text{min}}$ being the smallest eigenvalue of $L$, we have that
{\allowdisplaybreaks
\begin{align*}
    \| F_k(t) v \|^2_{V_{-1 + \varepsilon}^\alpha(k)} &= \sum_{\vec{j} \in \mathbb{N}^k} \pi(\lambda_{\vec{j}})^\alpha \sigma(\lambda_{\vec{j}})^{-1 + \varepsilon} \left| \sum_{l = 1}^k \langle -f, \varphi_{\vec{j}_l}\rangle \hat{v}(\skipindex{j}{l}) \right|^2 \\
    &\le \sum_{\vec{j} \in \mathbb{N}^k} \pi(\lambda_{\vec{j}})^\alpha \sigma(\lambda_{\vec{j}})^{-1 + \varepsilon} \left[\sum_{l = 1}^k
\lambda_{\vec{j}_l}^\alpha \langle f, \varphi_{\vec{j}_l}\rangle^2 \right] \left[\sum_{l = 1}^k \lambda_{\vec{j}_l}^{-\alpha} \hat{v}(\skipindex{j}{l})^2 \right] \\
    &= \sum_{\vec{j} \in \mathbb{N}^k} \underbrace{\sigma(\lambda_{\vec{j}})^{-1 + \varepsilon}}_{\le \lambda_{\text{min}}^{-1 + \varepsilon} k^{-1 + \varepsilon}} \left[\sum_{l = 1}^k \lambda_{\vec{j}_l}^\alpha \langle f, \varphi_{\vec{j}_l}\rangle^2 \right] \left[\sum_{l = 1}^k \pi(\lambda_{\vec{j}})^\alpha \lambda_{\vec{j}_l}^{-\alpha} \hat{v}(\skipindex{j}{l})^2 \right] \\
    &\le \lambda_{\text{min}}^{-1 + \varepsilon} k^{-1 + \varepsilon}\left[ \sum_{j \in \mathbb{N}} \sum_{l = 1}^k \lambda_{j}^\alpha \langle f, \varphi_{j}\rangle^2 \right] \left[\sum_{\vec{i} \in \mathbb{N}^{k - 1}} \sum_{l = 1}^k\pi(\lambda_{\vec{i}})^\alpha \hat{v}(\vec{i})^2 \right] \\
    &= \lambda_{\text{min}}^{-1 + \varepsilon} k^{1 + \varepsilon}\left[ \sum_{j \in \mathbb{N}} \lambda_{j}^\alpha \langle f, \varphi_{j}\rangle^2 \right] \left[\sum_{\vec{i} \in \mathbb{N}^{k - 1}} \pi(\lambda_{\vec{i}})^\alpha \hat{v}(\vec{i})^2 \right] \\
    &\le \lambda_{\text{min}}^{-1 + \varepsilon} k^{1 + \varepsilon} \| f(t) \|_{V^\alpha}^2 \sum_{\vec{i} \in \mathbb{N}^{k - 1}} \underbrace{\sigma(\lambda_{\vec{i}})^{-1}}_{\le \lambda_{\text{min}}^{-1} \frac{1}{k - 1}} \sigma(\lambda_{\vec{i}}) \pi(\lambda_{\vec{i}})^\alpha \hat{v}(\vec{i})^2 \\
    &\le \lambda_{\text{min}}^{-2 + \varepsilon} \frac{k^{1 + \varepsilon}}{k - 1} \|f(t) \|_{V^\alpha}^2 \|v \|_{V_1^\alpha(k)}^2 \\
    &\le 2 \lambda_{\text{min}}^{-2 + \varepsilon} k^\varepsilon \|f(t) \|_{V^\alpha}^2 \|v \|_{V_1^\alpha(k)}^2,
\end{align*}
}%
since $k \ge 2$. This concludes the proof.
\end{proof}

\end{document}